\documentclass[11pt]{article} 
\usepackage{boxedminipage}
\usepackage{amsfonts}
\usepackage{amsmath} 
\usepackage{amssymb}
\usepackage{graphicx}
\usepackage{amsthm}
\usepackage{t1enc}
\usepackage{subfig}
\usepackage{cite}
\usepackage[dvipsnames]{xcolor}
\usepackage{tikz}
\usepackage{tikz-cd}
\usepackage{stmaryrd}
\usepackage{mathtools}

\usepackage{bm}
\usepackage{marvosym}
\usepackage{bm}

\usepackage[utf8]{inputenc}
\usepackage{boondox-cal}
\usetikzlibrary{arrows.meta}

\newcommand{\xdashrightarrow}[1][0.97em]{\mathrel{
    \tikz[baseline]{\draw[dash pattern=on .25em off .1 em,->](0,.64ex)--(#1,.64ex)}}}
\newtheorem{theorem}{Theorem}
\newtheorem{lemma}[theorem]{Lemma}
\newtheorem{proposition}[theorem]{Proposition}
\newtheorem{corollary}[theorem]{Corollary}

\theoremstyle{definition}
\newtheorem{definition}{Definition}[section]
\newtheorem{example}{Example}[subsection]
\newtheorem{remark}{Remark}
\newcommand{\C}{\mathbb C}  
\newcommand{\N}{\mathbb N}  
\newcommand{\Z}{\mathbb Z}  
\newcommand{\A}{\mathbb A}

\newcommand{\Y}{\mathcal{Y}}

\title{}
\author{}
\date{}

\newcommand{\vect}{\mathbcal{V\hspace{-0.445em}e\hspace{-0.25em}c\hspace{-0.175em}t}}
\newcommand{\coalg}{\mathbcal{C\hspace{-0.295em}o\hspace{-0.3em}a\hspace{-0.215em}l\hspace{-0.375em}g}}
\newcommand{\set}{\mathbcal{S\hspace{-0.225em}e\hspace{-0.125em}t}}
\newcommand{\modk}{\mathbcal{M\!\!o\!\!d}_{\hspace{-0.095em}\Bbbk}}
\newcommand{\modB}{\mathbcal{M\!\!o\!\!d}_{\hspace{-0.175em}B}}

\newcommand{\cat}{\mathbcal{C}}
\newcommand{\alg}{\mathbcal{A\hspace{-0.175em}l\hspace{-0.375em}g}}
\newcommand{\comod}{\mathbcal{C\hspace{-0.295em}o\hspace{-0.2em}m\hspace{-0.285em}o\hspace{-0.275em}d}}
\newcommand{\rotvdash}{%
           \mathrel{\raisebox{.1em}{%
           \reflectbox{\rotatebox[origin=c]{90}{$\vdash$}}}}}

\newcommand{\End}{\mathbcal{E\hspace{-0.245em}n\hspace{-0.275em}d}}
\newcommand{\st}{\mathbcal{S\hspace{-0.2em}t}}

\begin{document}
\title{Introduction to Quantum Combinatorics}
\author{Tomasz Maszczyk}
\maketitle
\begin{abstract}
We construct a topos of quantum sets and embed into it the classical topos of sets. We show that the internal logic of the topos of sets, when interpreted in the topos of quantum sets, provides the Birkhoff-von Neumann quantum propositional calculus of idempotents in a canonical internal commutative algebra of the topos of quantum sets. We extend this construction by allowing the quantum counterpart of Boolean algebras of classical truth values which we introduce and study in detail. We realize expected values of observables in quantum states in our topos of quantum sets as a tautological morphism from the canonical internal commutative algebra to a canonical internal object of affine functions on quantum states. We show also that in our topos of quantum sets one can speak about quantum quivers in the sense of Day-Street and Chikhladze. Finally, we provide a categorical derivation of the Leavitt path algebra of such a quantum quiver and relate it to the category of stable representations of the quiver. It is based on a categorification of the Cuntz-Pimsner algebra in the context of functor adjunctions replacing the customary use of Hilbert modules.
\end{abstract}
\addcontentsline{toc}{section}{Acknowledgement}
\tableofcontents
\newpage

\section{Introduction}

The history of the development of C*-algebras associated with combinatorial data began in the context of a shift of finite type, which was associated by Cuntz and Krieger \cite{C-K-80}with a transition matrix. It was generalized by various authors, including Bates, Fowler, Kumjian, Laca, Pask, and Raeburn \cite{F-L-R-00}, \cite{K-P-R-R-97}, to the context of more general subshifts associated with directed graphs. In another direction, Exel and Laca \cite{E-L-99} partially generalized Cuntz–Krieger algebras by allowing for an infinite matrix with 0 and 1 entries. These two approaches were unified by Tomforde \cite{T-03} in the context of hypergraphs.
Next, Carlsen \cite{C-08}, using the so-called extended Matsumoto’s construction \cite{M-97}, addressed the case of a general one-sided subshift over a finite alphabet, while Bates and Pask \cite{B-P-07}  associated  so-called labeled graph algebras with any shift space over a finite alphabet. Through the work of Carlsen, Ortega, and Pardo \cite{C-O-P-17}, it was shown that the case of labeled graphs can be realized as a Boolean algebra of a family of subsets of vertices, with partial actions given by the arrows. All these cases represent the 0-dimensional scenario of topological graphs, as explored by Katsura \cite{K-03}\cite{K-04}\cite{K-04Trans}, where the topology of actions can be reduced to their combinatorial structure.
The approach of Kumjian, Pask, Raeburn, and Renault \cite{K-P-R-R-97} regarding the groupoid presentation of graph C*-algebras has been extended by Exel \cite{E-08} in the direction of presenting quite general combinatorial C*-algebras via groupoids or inverse semigroups. Such presentations are crucial for studying properties of these C*-algebras, such as ideal structure, simplicity, and purely infinite properties, as well as for the computation of their K-theory.

Inspired by this development and motivated by Quantum Information Theory, efforts to extend these concepts to the context of quantum graphs appeared recently. Similarly, the initial datum is the adjacency matrix, quantized in an appropriate sense. Motivated by the work of Musto-Reutter-Verdon \cite{M-R-V-18}, it associates a quantum version of the Cuntz-Krieger C*-algebra with these structures. However, classically, this only covers the case of graphs with a finite set of vertices and without multiple edges. Moreover, it is not clear what a quantum counterpart of combinatorics might be in this context, nor how it relates to Boolean algebras, symmetries, and dynamical systems defined by (partial) semigroup actions. 
 
In addition, all such problems are convoluted because of a quite rich context, especially the presence of a dagger structure related to the antilinear involutive antihomomorphism, C*-algebras, Hilbert modules and hermitian scalar products together with a probabilistic interpretation of a paring between states and projections.  Therefore, to clarify the situation,  it would be helpful to take step back and start accumulating structures according to the usual classical mathematical hierarchy which starts from the topos of sets, the classical logic as an internal logic of this topos etc.  As it will be shown here, surprisingly many constructions in the quantum paradigm arises on the fundamental categorical level.  The present paper begins with the Yoneda full embedding of the category of counital coassociative coalgebras over a field which, albeit being complete and cocomplete \cite{A-11Cat, A-11Lim}, is not a topos, into its presheaf topos which we call the topos of quantum sets. Better still, we construct a partial-monoidal structure on the Yoneda functor, relating the tensor product of coalgebras with the argument-wise cartesian product of presheaves. This connects a category with a structure of a topos, in which we can speak the classical language,  with a category in which we can speak the quantum language  based on coalgebraic calculus. While the linearization functor fully embeds the category of sets into the category of coalgebras, which allows us to emulate classical combinatorics within the category of coalgebras, the presence of non-cocommutative coalgebras is what allows us to model quantum combinatorics. Although linearization functor is strong monoidal as transforming the cartesian product of sets into the tensor product of coalgebras, and then composed with the Yoneda functor goes to the argument-wise cartesian product of presheaves on coalgebras, when extended to the tensor product of coalgebras, it encounters a non-trivial obstruction of an abstract Eckmann-Hilton-type , to simultaneous realizability of a pair of coalgebra maps from a coalgebra to two others coalgebras as a single coalgebra map to their tensor product. Our first result says that this condition defines a subfunctor of the argument-wise cartesian product  of representable presheaves satisfying axioms of a partial-associative product. It is then natural to speak about a (partial-associative) quantum Cartesian product of  quantum sets. Note that in opposite to the classical Cartesian product of sets, its quantum counterpart is not a categorical product, but only a partial-monoidal structure on a (sub)category of quantum sets. 
This allows us to construct quantum logic which, in  particular, contains the Birkhoff–von Neumann's quantum propositional calculus interpreted through the image of classical logic in quantum logic internal to the topos of quantum sets. In this internal language,  the problems intractable by other approaches to quantum combinatorics can be sensibly addressed. In particular, quantum Boolean algebras, quantum semigroups and their partial actions, quantum graphs, and their Leavitt path algebras can be approached  in this alternative way. The characteristic feature of all these algebraic constructions is that they, as presheaves on the category of coalgebras, take values in partial algebraic structures. When, using the linearization functor, we fully embed the category of sets into the category of quantum sets, all these constructions restrict to classical ones. 
The language of coalgebras clarifies several technical problems encountered in the dual language of topological algebras when trying to relate noncommutative theory with classical combinatorics. In the situation of (quantum, row finite) graphs, usually people apply the Cuntz-Pimsner construction for a graph-induced self-correspondence between vertices of the graph to associate with it a graph C*-algebra.  In the parlance of noncommutative geometry, it can be inerpreted as associating to a combinatorial datum a (locally compact Hausdorff) quantum  space and relate combinatorics to the thus induced noncommutative topology. In our approach, we categorify the Cuntz-Pimsner construction \cite{P-97} \cite{K-03}\cite{K-04}\cite{K-04Trans} for a self-correspondence,  to an exact (dualizable) endofunctor on the category of comodules over the coalgebra object of vertices of a quantum graph, to associate with a  graph its Leavitt path monad. We prove that the Eilenberg-Moore category of this monad is equivalent to a category of quiver representations equipped with some retraction. When in this result the latter category is an analog of the Yetter-Drinfeld module category for a (finite dimensional) Hopf algebra, the Leavit path monad is an analog of the Drinfeld double algebra whose modules form a category equivalent to  the Yetter-Drinfeld module category. Since every graph C*-algebra is a C*-completion of a Leavitt path algebra, it gives an insight into the deep categorical roots of the  bicategorical interpretation of Meyer-Sehnem \cite{Mey-Seh-19} of the Cuntz-Pimsner construction. In opposite to their interpretation which can be regarded as viewing it from above and in the rich C*-algebraic and Hilbert module context, we propose a complementary derivation of it from the bottom up explaining it as coming from canonical adjunctions associated with a (quantum) quiver already on the (quantum) combinatorial level. We believe that understanding these categorical point of view could be helpful in future more and more categorical questions posed on the way of developing the theory also in the context of C*-algebras.  The point is that the construction of a graph C*-algebra of a quiver, in the parlance of Noncommutative Geometry, associates with a quiver a locally compact quantum space represented by a C*-algebra, while we associate with it only a quantum-combinatorial categorical object. Since both the quiver and the resulting object are both combinatorial in nature, except one can be entirely classical and the other is inherently quantum, difficulties created by a C*-algebraic structure cannot even arise. Therefore (quantum) combinatorial laboratory offers an opportunity of easy testing categorical questions before they are to be generalized to the algebraic or C*-algebraic context.

\section{Quantum Sets}
\subsection{The topos of quantum sets}
Let  $\Bbbk$ be a fixed field
of characteristic zero. By $\vect$ we denote the symmetric monoidal category of $\Bbbk$-vector spaces under the tensor product over $\Bbbk$, and by  $\coalg$  the category of counital coassociative coalgebras over $\Bbbk$, i.e. comonoids in $\vect$. In calculations for coalgebras we use the Heyneman-Sweedler convention \cite{S-69}, e.g. for the comultiplication $\Delta: C\rightarrow C\otimes C$ we use the notation $\Delta(c)=c_{(1)}\otimes c_{(1)}$.

    \vspace{4em}
    \begin{definition}
         The {\bf topos of quantum sets} $q\set$     
     is the topos of  presheaves on
$\coalg$ with natural transformations as morphisms.  
It is a symmetric monoidal category  with respect to an argument-wise product  as a monoidal product and the singleton-constant presheaf as a monoidal unit.
    \end{definition}

Note that  the coalgebras of rank zero and one represent presheaves which are the initial and the terminal objects of $q\set$\!, respectively. Moreover, 
the initial object is monoidally absorbing and the terminal one is a monoidal unit. These we will call the  {\bf quantum empty set} and the {\bf quantum singleton} and, for the reasons which will be explained in subsection \ref{set}, denote classically as $\emptyset$ and $\{\bullet\}$, respectively.

  \begin{definition}
   The  internal language and the internal logic \cite{ML-M-97} of the topos $q\set$     
will be called {\bf quantum language} and  {\bf quantum logic}, respectively.
    \end{definition}
  \begin{definition}
   A quantum set is called {\bf representable} if it is representable as a presheaf \cite{ML-M-97}.  
    \end{definition}

\subsubsection{Quantum elements}

    \begin{definition}
For every quantum set $X$ we define the category of {\bf quantum elements}, whose objects are pairs $(C, x)$ where $C$ is a finitely dimensional coalgebra, $x\in X(C)$, and morphisms $(C', x')\rightarrow (C, x)$ being coalgebra maps $\pi: C'\rightarrow C$ such that $X(\pi)(x)=x'$, with an obvious composition.
     \end{definition}

This definition is based on the general notion of the \emph{category of elements} of a presheaf \cite{ML-M-97}.   The \emph{Fundamental Theorem on Coalgebras} \cite{S-69} saying that every coalgebra is a colimit of finite dimensional coalgebras, the fact that finite dimensional ones form a \emph{small subcategory of compact objects}, and finally the \emph{Density Theorem} \cite{ML-M-97} saying that every presheaf on a small category is a colimit of representable ones, altogether can be summarised to what in this context should be called the {\bf Quantum Extensionality Principle}.

\subsubsection{Quantum elements in classical mathematics}
\begin{example}
    For $S = \{ \bullet \}$ (a singleton), we recover the set $T$ as the set of its classical elements
    $$\Y (\Bbbk T)(\Bbbk \{ \bullet \}) =  T.$$

If $T=\emptyset$ (empty set), then $\Y (\Bbbk \emptyset)$ is an initial object in $q\set $.

If $T=\{ \bullet \}$ (a singleton), $\Y (\Bbbk \{ \bullet \})$ is a terminal object in $q\set$.

By virtue of Theorem \ref{setqset} it means that the singleton regarded as a quantum set (quantum singleton) has a single quantum element. Since this element is classical, the quantum singleton is an example of a quantum set with no non-classical elements.

\end{example}

\begin{example}
{\bf  Group-like elements.} Let $\Bbbk S$ be the $\Bbbk-$linearization of a set $S$ with its standard coalgebra structure.  For $S = \{ \bullet \}$ (a singleton), we obtain $\Y (D)(\Bbbk \{ \bullet \}) = G(D)$, \emph{the set of gruplike elements} of the coalgebra $D$. They can be interpreted as classical elements of the quantum set $\Y (D)$. In general,  $$\Y (D)(\Bbbk S) = \Big\{g_s \in G(D)\ \Big| \ s \in S\Big\},$$  is a family of grouplike elements of the coalgebra $D$, i.e.
    $$ \Delta (g_{s}) = g_{s} \otimes g_{s}, 
    \quad \varepsilon(g_s) = 1,$$
    indexed by the set $S$.
\end{example}

\begin{example}{\bf Primitive elements.} Let $C = \Bbbk \{\bullet,  \text{\Flatsteel}\hspace{-0.45em} \blacktriangleright  \}$
    be a coalgebra with 
    $$\Delta(\bullet) = \bullet \otimes \bullet, \quad  
    \Delta(\text{\Flatsteel}\hspace{-0.45em} \blacktriangleright ) = \bullet \hspace{0.25em} \otimes \text{\Flatsteel}\hspace{-0.45em} \blacktriangleright   +\  \text{\Flatsteel}\hspace{-0.45em} \blacktriangleright  \otimes \hspace{0.25em}\bullet,$$
    $$\varepsilon(\bullet) = 1, \quad \varepsilon(\text{\Flatsteel}\hspace{-0.45em} \blacktriangleright  ) = 0$$

    The canonical injective coalgebra map $\Bbbk \{\bullet \} \rightarrow \Bbbk \{\bullet,  \text{\Flatsteel}\hspace{-0.5em} \blacktriangleright  \}$ induces a surjective map on quantum elements
    $$\Y (D)(\Bbbk \{\bullet, \text{\Flatsteel}\hspace{-0.5em} \blacktriangleright  \}) \rightarrow \Y (D)(\Bbbk \{ \bullet \})$$
which  coincides with the canonical projection
\begin{align*}
\pi: \coprod_{g\in G(D)}P_{g}(D)\rightarrow G(D),
\end{align*}
where the preimage of an element $g \in G(D)$
    $$\pi ^{-1} (g) = P_g(D):=\Big\{ p \in D \ \Big| \ \Delta(p) = p\otimes g + g \otimes p, \quad \varepsilon(p) = 0 \Big\}$$
    is the subset of $g-$primitive elements of $D$. 

    In particular, a group and a Lie algebra can be seen as elements of (representable) quantum sets as follows.
    \begin{enumerate}
        \item $\Y(\Bbbk G)(\Bbbk\{\bullet\}) = G$
        \item $\Y(U(\mathfrak{g}))(\Bbbk\{\bullet,\text{\Flatsteel} \hspace{-0.45em} \blacktriangleright \}) = \mathfrak{g}$
    \end{enumerate}
\end{example}

\begin{example}{\bf Quantum elements of finite decomposition categories}
    \begin{definition}
        If every morphism in a category $\mathbcal{D} $ has finitely many decompositions, we say that this category has {\bf  finite decomposition} property, f.d. in short.
    \end{definition}

For any small category with the finite decomposition property the linearization of the set of morphisms $D:=\Bbbk \mathbcal{D}$ admits a coalgebra structure with
\begin{equation*}
                \Delta(m) = \sum_{m_1\circ m_2 = m} m_1 \otimes m_2, \qquad \varepsilon(m) = \begin{cases}
                1 \qquad\text{$m$ an identity}\\
                0 \qquad\text{otherwise}
                \end{cases}
\end{equation*}

\begin{proposition}
For every small category with f.d. property and $D$ as above,
\begin{equation*}
\Y(D)(C) = 
\{ 
x: \mathbcal{D}\rightarrow  F(C) 
\   \mid \ x\  \text{\rm  a   {\bf functor of finite support,}}
\  \sum_{\mathbcal{d}\in {\rm Ob(\mathbcal{D})}}x_{\hspace{-0.15em}\mathbcal{d}}=1
\}. 
\end{equation*}
where by $F(C)$ we mean its multiplicative monoid with zero and we identify objects with identity morphisms.
\end{proposition}

\begin{proof}
Writing $x(c)=\sum_{m}x_{m}(c) m$ and  regarding $m\mapsto x_{m}\in F(C)$ as a function with finite support, we see that $x$ being compatible with comultiplication and the counit is tantamount to 
$$x_{m_1\circ m_2} = x_{m_1}x_{m_2},\ \ \  \sum_{\mathbcal{d}\in {\rm Ob(\mathbcal{D})}}x_{\hspace{-0.15em}\mathbcal{d}}=1,  $$
respectively. 
\end{proof}
Note that for a discrete category, we obtain the complete system of orthogonal idempotents, as before.

Here are the three most important examples of categories with finite decomposition. In the first two we regard a monoid  as a category with a single object.
\begin{enumerate}
\item  $\mathbcal{C} = (\N, +, 0)$.   
$\quad D = \Bbbk \Big\{ d_{n} \  \Big| \ n \in \N \Big\},$
    $$\Delta : d_n \rightarrow \sum_{n' + n'' = n} d_{n'}\otimes d_{n''}, \qquad\varepsilon(d_n) = \delta_{n,0},$$
    $$\Y (D)(C) = \Big\{  x\in F(C)\ \Big|\  x\ {\bf nilpotent}\Big\}.$$
Here $x$ is the coefficient of the coalgebra map $C\rightarrow D$ at $d_1$. Since $F(D)\cong \Bbbk[[t]]$ is the algebra of formal power series, and $\Y (D)$ admits a distinguished element $\{ \bullet\}\rightarrow \Y (D)$ corresponding to $d_{0}$, dual to the augmentation $F(D)\cong \Bbbk[[t]]\rightarrow \Bbbk$, $t\mapsto 0$, $\Y (D)$  can be regarded as a base parameterizing formal  deformations of quantum sets. In Section ??, it will be used to define formal deformation quantisations of classical sets.
\item $\mathbcal{C} = (\N_{>0}, \times, 1), \quad D = \Bbbk \Big\{ d_{n} \  \Big| \ n \in \N_{>0} \Big\},$
    $$\Delta : d_n \rightarrow \sum_{n' \times n'' = n} d_{n'}\otimes d_{n''}, \qquad\varepsilon(d_n) = \delta_{n,1},$$
    $$\Y (D)(C) = \Big\{ \{x_p\}_{p\in P }\subset F(C)\  \Big| \  \{x_p\}_{p\in P }\ {\bf finite\  family\  of\  commuting\  nilpotents} \Big\}$$
Here $x_p$  is the coefficient of the coalgebra map $C\rightarrow D$ at $d_p$,  while  $P$ is a finite set of primes.
\item $\mathbcal{C}=I\times I$, where $I$ is a finite set, with objects $i\in I$ and arrows $(i, j)\in I\times I$ (we adopt the convention saying that $(i, j)$ is an arrow from $j$ to $i$), the composition  $(i, j)\circ  (j, k):=(i, k)$ and identities $(i, i)$, is called the {\bf pair category} on $I$. Then
    $$D := \Bbbk \Big\{ d_{i,j} \  \Big| \ i,j \in \left\{1, \ldots, n \right\} \Big\},$$
    $$ \Delta(d_{i,k}) = \sum_j d_{i, j} \otimes d_{j, k}, \quad \varepsilon(d_{i, k}) = \delta_{i, k} $$
is called the {\bf comatrix coalgebra} and represents a quantum set  of systems of {\bf matrix units} in $F$, i.e. 
$$\Y (D)(C) = \Big\{ \{x_{ij}\}_{\ i,j\in I }\  \Big| \  x_{ij}x_{j'k}=\delta_{j, j'}x_{ik}, \sum_{ i}x_{ii} = 1\ {\rm in}\  F(C).\Big\}.$$
\end{enumerate}
\end{example}   

\begin{example}{\bf Finite-dimensional group representations as quantum elements.}

    For an affine group scheme $G$ over $\Bbbk$, and a finitely dimensional $\Bbbk-$vector space $V$ we can form a set $Rep_\Bbbk(G, V)$ of $\mathcal{O}(G)-$comodule structures on $V$. Then
$${\rm Rep}_\Bbbk(G, V) = \Y (\mathcal{O}(G))(V^* \otimes V).$$
\end{example}

\subsubsection{Quantum Cartesian product}
  By the Yoneda embedding theorem \cite{ML-M-97} , $\coalg$ and  $q\set$ are related by the fully faithful Yoneda functor
$$\coalg \xrightarrow{\Y} {\rm PSh}(\coalg)=: q\set,$$
$$\Y(D) = \coalg(-, D).$$
   Since both categories 
$$\big(\coalg, \otimes, \Bbbk \{\bullet\}\big)\quad \text{and}\quad  \big({\rm PSh}(\coalg), \times, \{\bullet\}\big)$$
   are {\bf symmetric monoidal} 
   one could ask how the Yoneda embedding interacts with those structures.
To answer this question,  we need Segal's notion of a partial monoid \cite{S-73} generalized and adapted to our categorical context.

\begin{definition}
We call a full subcategory $\cat$ of $q\set$ {\bf partial monoidal} if it contains a monoidal unit $\{\bullet\}$ and for any two objects $X_1$,    $X_2$  of $\cat$ in this subcategory there is a sub-presheaf $X_1\times_{\cat}X_2$ in $\cat$ of the argument-wise product $X_1\times X_2$ of presheaves in $q\set$, i.e.
$$\big( X_1\times_{\cat}X_2\big)(C)\subseteq X_1(C)\times X_2(C),$$
\end{definition} 
\noindent such that for any three objects $X_1$, $X_2$, $X_3$ of $\cat$
\begin{equation}\label{(1,2),3}
\begin{split}
(x_1, x_2 )\in ( X_1\times_{\cat}  X_2 )(C),\ \ 
(( x_1, x_2 ), x_3)\in ((X_1\times_{\cat}  X_2 )\times_{\cat} X_3)(C)
\end{split}
\end{equation}
if and only if 
\begin{equation}\label{1,(2, 3)}
(x_2, x_3 )\in ( X_2\times_{\cat}  X_3 )(C), \ \ 
(x_1, (x_2 , x_3))\in (X_1\times_{\cat}  (X_2 \times_{\cat} X_3))(C),
\end{equation}
and then, under the identification
$$\big( X_1(C)\times X_2(C) \big)\times X_3(C) \cong   X_1(C)\times\big( X_2(C)\times  X_3(C)\big), $$
the equality
\begin{align}\label{assoc}
(( x_1, x_2 ), x_3) =  (x_1, (x_2 , x_3))
\end{align}
holds, and for every object $X$ in $\cat$ the containments $\{\bullet\}\times_{\cat}X \subseteq \{\bullet\}\times X $ and $X \times_{\cat} \{\bullet\} \subseteq X \times\{\bullet\}$ are equalities.

\begin{theorem}\label{sub}
The full subcategory $\cat$ of representable quantum sets is a partial monoidal symmetric subcategory of  $q\set$.
\end{theorem}

\begin{proof}
For $X_{i}:=\Y(D_{i})$, we define $( X_1\times_{\cat}  X_2 )(C)$ as
\begin{equation}\label{def(1, 2)}
\{ 
(x_1, x_2)\in  X_1(C)\times  X_2 (C)\ \mid\ 
x_1(c_{(1)})\otimes x_2(c_{(2)})
=x_1(c_{(2)})\otimes x_2(c_{(1)})
\in D_{1}\otimes D_{2}
\}.
\end{equation}
To prove that $X_1\times_{\cat}  X_2 $ is indeed representable, we define
\begin{align}\label{1, 2}
(x_1, x_2)(c):=x_1(c_{(1)})\otimes x_2(c_{(2)})
\end{align}
and show that it defines a coalgebra map $(x_1, x_2): C\rightarrow D_1\otimes D_2$.

First, using (\ref{1, 2}),  (\ref{def(1, 2)}), coassociativity, the fact that $x_i: C\rightarrow D_{i}$ are coalgebra maps and the diagonal comultiplication of $D_1\otimes D_2$, we check that $(x_1, x_2)$ respects comultiplication
\begin{align}
  \  & (x_1, x_2)(c_{(1)})\otimes (x_1, x_2)(c_{(2)})\\
 =\ & x_1(c_{(1)(1)})\otimes  x_2(c_{(1)(2)})\otimes  x_1(c_{(2)(1)})\otimes  x_2(c_{(2)(2)})\\
=\ & x_1(c_{(1)})\otimes  x_2(c_{(2)})\otimes  x_1(c_{(3)})\otimes  x_2(c_{(4)})\\
=\ & x_1(c_{(1)})\otimes  x_2(c_{(2)(1)})\otimes  x_1(c_{(2)(2)})\otimes  x_2(c_{(3)})\\
=\ & x_1(c_{(1)})\otimes  x_2(c_{(2)(2)})\otimes  x_1(c_{(2)(1)})\otimes  x_2(c_{(3)})\\
=\ & x_1(c_{(1)})\otimes  x_2(c_{(3)})\otimes  x_1(c_{(2)})\otimes  x_2(c_{(4)})\\
=\ & x_1(c_{(1)(1)})\otimes  x_2(c_{(2)(1)})\otimes  x_1(c_{(1)(2)})\otimes  x_2(c_{(2)(2))})\\
=\ & x_1(c_{(1)})_{(1)}\otimes  x_2(c_{(2)})_{(1)}\otimes  x_1(c_{(1)})_{(2)}\otimes  x_2(c_{(2)})_{(2)}\\
=\ & (x_1(c_{(1)})\otimes  x_2(c_{(2)})_{(1)}\otimes  (x_1(c_{(1)})\otimes  x_2(c_{(2)})_{(2)}\\
=\ & (x_1, x_2)(c)_{(1)}\otimes (x_1, x_2)(c)_{(2)}.                                                              
\end{align}
Next, using (\ref{1, 2}),  the diagonal counit of $D_1\otimes D_2$,  the fact that $x_i: C\rightarrow D_{i}$ are coalgebra maps and counitality, we check that it respects the counit
\begin{align}
  \  & \varepsilon\left((x_1, x_2)(c)\right)\\
 =\ & \varepsilon\left(x_1(c_{(1)})\otimes x_2(c_{(2)}) \right)\\
=\ & \varepsilon\left(x_1(c_{(1)})\right) \varepsilon\left(x_2(c_{(2)}) \right)\\
=\ & \varepsilon\left(c_{(1)}\right) \varepsilon\left(c_{(2)} \right)\\
=\ & \varepsilon\left(c_{(1)}\varepsilon(c_{(2)}) \right)\\
=\ & \varepsilon(c).                                                           
\end{align}

Now, we are to prove that the conditions  \eqref{(1,2),3} and  \eqref{1,(2, 3)} are satisfied. First, we use the definition \eqref{def(1, 2)} to rewrite them as follows
$$x_1(c_{(1)})\otimes x_2(c_{(2)})=x_1(c_{(2)})\otimes x_2(c_{(1)}),$$ 
$$(x_1, x_2)(c_{(1)})\otimes x_3(c_{(2)}) = (x_1, x_2)(c_{(2)})\otimes x_3(c_{(1)})$$
and
$$x_2(c_{(1)})\otimes x_3(c_{(2)})=x_2(c_{(2)})\otimes x_3(c_{(1)}),$$ 
$$x_1(c_{(1)})\otimes (x_2, x_3)(c_{(2)}) = x_1(c_{(2)})\otimes (x_2, x_3)(c_{(1)}).$$
By (\ref{1, 2}) they can be rewritten as
$$x_1(c_{(1)})\otimes x_2(c_{(2)})=x_1(c_{(2)})\otimes x_2(c_{(1)}),$$ 
$$x_1((c_{(1)(1)}))\otimes x_2(c_{(1)(2)})\otimes x_3(c_{(2)}) = x_1(c_{(2)(1)})\otimes x_2(c_{(2)(2)})\otimes x_3(c_{(1)}),$$
and
$$x_2(c_{(1)})\otimes x_3(c_{(2)})=x_2(c_{(2)})\otimes x_3(c_{(1)}),$$ 
$$x_1(c_{(1)})\otimes x_2(c_{(2)(1)})\otimes x_3(c_{(2)(2)}) = x_1(c_{(2)})\otimes x_2c_{(1)(1)})\otimes x_3(c_{(1)(2)}).$$

By applying the counit to the utmost left tensor factor in the triple tensor in the first pair of identities and to the utmost right tensor slot in the triple tensor of the second pair, using the fact that $x_i$ are  coalgebra maps and the Heyneman-Sweedler convention we get the following two systems of identities
$$x_1(c_{(1)})\otimes x_2(c_{(2)})=x_1(c_{(2)})\otimes x_2(c_{(1)}),$$
$$x_2(c_{(1)})\otimes x_3(c_{(2)})=x_2(c_{(2)})\otimes x_3(c_{(1)}),$$
$$ x_1((c_{(1)}))\otimes x_2(c_{(2)})\otimes x_3(c_{(3)}) = x_1(c_{(2)})\otimes x_2(c_{(3)})\otimes x_3(c_{(1)}),$$
and
$$x_2(c_{(1)})\otimes x_3(c_{(2)})=x_2(c_{(2)})\otimes x_3(c_{(1)}),$$
$$x_1(c_{(1)})\otimes x_2(c_{(2)})=x_1(c_{(2)})\otimes x_2(c_{(1)}),$$
$$ x_1(c_{(1)})\otimes x_2(c_{(2)})\otimes x_3(c_{(3)}) = x_1(c_{(3)})\otimes x_2c_{(1)})\otimes x_3(c_{(2)}).$$

To prove equivalence of the thus rewritten two systems, observe that first two identities in both systems differ only by their order. hence it is enough to prove that assuming them allows us to make the right hand sides of both remaining identities  equal. To this end, it is enough to use those first two identities and coassociativity in the Heyneman-Sweedler convention as follows
\begin{align}
   \ & x_1(c_{(2))})\otimes x_2(c_{(3)})\otimes x_3(c_{(1)})\\
 =\ & x_1(c_{(2)(1))})\otimes x_2(c_{(2)(2)})\otimes x_3(c_{(1)})\\
=\ & x_1(c_{(2)(2))})\otimes x_2(c_{(2)(1)})\otimes x_3(c_{(1)})\\
=\ & x_1(c_{(3)})\otimes x_2(c_{(2)})\otimes x_3(c_{(1)})\\
=\ & x_1(c_{(2)})\otimes x_2(c_{(1)(2)})\otimes x_3(c_{(1)(1)})\\
=\ & x_1(c_{(2)})\otimes x_2(c_{(1)(1)})\otimes x_3(c_{(1)(2)})\\
=\ & x_1(c_{(3)})\otimes x_2(c_{(1)})\otimes x_3(c_{(2)}).                                                         
\end{align}

Finally, assuming that this identities hold, we check (\ref{assoc}) using (\ref{1, 2}) and associativity in the Heyneman-Sweedler convention as follows.
\begin{align}
  \  & \left((x_1, x_2), x_3\right)(c)\\
 =\ & (x_1, x_2)(c_{(1)})\otimes x_3(c_{(2)})\\
=\ & x_1(c_{(1)(1)})\otimes x_2(c_{(1)(2)})\otimes x_3(c_{(2)})\\
=\ & x_1(c_{(1)})\otimes x_2(c_{(2)(1)})\otimes x_3(c_{(2)(2)})\\
=\ & x_1(c_{(1)})\otimes (x_2, x_3)(c_{(2)})\\
=\ & (x_1, (x_2, x_3))(c).                                                           
\end{align}

Now, to check axioms related to the monoidal unit, observe that for $X=\Y(D)$, the presheaves $\{\bullet\}\times_{\cat} X$ and $X\times_{\cat}\{\bullet\}$ evaluated on $C$ become both $X(C)$ by the counit identities, since the only element of $\{\bullet\}(C)=\Y(\Bbbk\{\bullet\})(C)$ is the counit of $C$.

\end{proof}
\begin{corollary}
  The Yonneda embedding of $\coalg$ into $q\set = {\rm PSh}(\coalg)$ admits a structure of a {\bf  symmetric partial monoidal} functor.
\end{corollary}

\begin{proof}
Let $\cat$ be the full subcategory of representable presheaves in $q\set$. It is the essential image of the Yonneda embedding $\Y$.
   First, note that the compatibility of {\bf monoidal units} morphism under $\Y$
    $$\{ \bullet \} \xrightarrow{\simeq}\Y\big(\Bbbk \{ \bullet \}\big)$$ 
    is {\bf total}, where its $C-$component reads as
$$\{\bullet\}(C) = \{\bullet\} \rightarrow \Y(\Bbbk\{\bullet\})(C) = \coalg(C, \Bbbk\{\bullet\}),$$
$$\bullet \mapsto \varepsilon_C.$$
However, the {\bf binatural transformation} 
$$\Y(D_1) \times \Y(D_2) \rightarrow \Y(D_1 \otimes D_2)$$
is only {\bf partial} with the domain $\Y(D_1) \times_{\cat} \Y(D_2)$, whose $C-$component reads as
$$ \big(\Y(D_1) \times_{\cat} \Y(D_2)\big)(C) \rightarrow \Y(D_1 \otimes D_2)(C),$$ 
i.e. consists of pairs 
$$(x_1, x_2) \in (\Y(D_1) \times_{\cat} \Y(D_2))(C)\subseteq \Y(D_1)(C) \times \Y(D_2)(C) 
$$
such that the map
$$ (x_1, x_2):=\big(c \mapsto x_1(c_{(1)}) \otimes x_2(c_{(2)})\big) \in \vect(C, D_1\otimes D_2)$$ 
belongs in fact to 
$\coalg(C, D\otimes D').$

Checking compatibility with associativity and symmetry, based on Theorem \ref{sub}, is left to the reader.
\end{proof}

\subsection{Sets as quantum sets}\label{set}
The following theorem, relating topos of sets and the topos of quantum sets, could be used to model a classical observer in the quantum world.
\begin{theorem}
The composition of the functor of linearization and the Yoneda embedding
 \begin{align*}
\set\xhookrightarrow{\ \ \Bbbk -\ \ }\coalg \xhookrightarrow{\ \ \Y\ \  }{\rm PSh}(\coalg)=: q\set
\end{align*} 
is a symmetric partial-monoidal full embedding.
\end{theorem}

\begin{proof}
First, we observe that the linearization functor $\Bbbk -: \set\hookrightarrow\coalg$ is
symmetric strong monoidal when relating  the cartesian product of sets with the singleton  as a monoidal unit with the tensor product of coalgebras with a final coalgebra as a monoidal unit. To prove that it is fully faithful, we compute the set of morphisms $\coalg(\Bbbk S, \Bbbk T)$ to be a set of maps $S\rightarrow \Bbbk T$, $s\mapsto \sum_{t\in T}x_{t}(s)t$ where
\begin{align}\label{alphacomult}
\begin{split}
x_{t_1}(s)x_{t_2}(s) & =\delta_{t_1, t_2}x_{t_1}(s),\\
\sum _{t\in T}x_{t}(s) & = 1.
\end{split}
\end{align}
Since ${\rm char}(\Bbbk) = 0$ and $\Bbbk$ has no nontrivial idempotents, for every $s\in S$ there exists a unique $t_{0}\in T$ such that $x_{t}(s)=\delta_{t, t_0}$ and then such  $\alpha$ is a unique solution to (\ref{alphacomult}). Let us define $f: S\rightarrow T$ so that $f(s):=t_0$. This is the natural inverse to the map   $\set(S, T)\rightarrow\coalg(\Bbbk S, \Bbbk T)$, $f\mapsto (s\mapsto f(s)= \sum _{t\in T}\delta_{t, f(s)}.$ Since the Yoneda embedding is also fully faithful, the composition is fully faithful as well. Since the linearization functor is (strong) monoidal and the Yoneda functor is partial-monoidal,  their composition is partial-monoidal as well. Checking compatibility with symmetry is left to the reader.
\end{proof}
 
This partial-monoidal full embedding $\set\hookrightarrow q\set$ can be described as follows in terms of  a distinguished internal algebra $F$  in the topos  $q\set$ where $F(C):=\vect(C, \Bbbk)$ is the dual convolution algebra of the coalgebra $C$ being its predual.   Since for every set $S$ the set $F(\Bbbk S)=\set(S, \Bbbk)$ is the algebra of $\Bbbk$-valued functions on $S$ with the argument-wise multiplication and the constant 1 as a unit, we call $F$ the {\bf internal quantum function algebra}. Note that it is a purely algebraic analog of the von Neumann algebra, since a von Neumann C*-algebra is defined as one having the predual \cite{S-97}. Note also that, by the existence and the universal property of the  cofree coalgebra ${\rm C}(\Bbbk)$ on the vector space $\Bbbk$ \cite{S-69}, $F$ is representable by ${\rm C}(\Bbbk)$, i.e. $F=\Y({\rm C}(\Bbbk))$.

\begin{proposition}\label{partcomm}
The internal quantum function algebra $F$ is partial commutative,  i.e. the restricted multiplication 
\begin{align}\label{partcommrest}
F\times_{\cat}F\rightarrow F,
\end{align}
where $\cat$ is the partial monoidal full subcategory in $q\set$ consisting of representable presheaves, is commutative.
\end{proposition}

\begin{proof}
Take $(x_1, x_2)\in (F\times_{\cat}F)(C)$. By (\ref{def(1, 2)}) and (\ref{1, 2}),  
\begin{align}\label{partcommrest}
x_1(c_{(1)})\otimes x_2(c_{(2)})
=x_1(c_{(2)})\otimes x_2(c_{(1)})
\in {\rm C}(\Bbbk)\otimes {\rm C}(\Bbbk).
\end{align}
Applying the universal linear map ${\rm C}(\Bbbk)\rightarrow \Bbbk$ to both slots of the tensor product on either side and then the commutative multiplication in $\Bbbk$ we get, regarding now coalgebra maps $x_i:C\rightarrow {\rm C}(\Bbbk)$ as linear maps $x_i:C\rightarrow \Bbbk$,
\begin{align}\label{comm}
x_1(c_{(1)}) x_2(c_{(2)})
=x_1(c_{(2)}) x_2(c_{(1)}) = x_2(c_{(1)}) x_1(c_{(2)})
\end{align}
in $\Bbbk$ which proves that in $F(C)$
\begin{align}\label{commF}
x_1 x_2 = x_2 x_1.
\end{align}
\end{proof}

\begin{proposition}\label{setqset}
The image of a set $T$ under the above embedding $\set\hookrightarrow q\set$  is the quantum complete set of orthogonal idempotents in $F$ indexed by $T$, i.e.
\begin{align}\label{ortidem}
\Y(\Bbbk T)(C)=
\Big\{ \{ x_t\}_{t\in T}\subset F(C)\ \Big|\  x_{t_1}x_{t_2} =\delta_{t_1, t_2}x_{t_1},\ \sum _{t\in T}x_{t} = 1\Big\}.
\end{align}
\end{proposition}
\begin{proof}
The proof repeats the proof of (\ref{alphacomult}) with $\Bbbk S$ replaced by an arbitrary $C$.
\end{proof}
It does make sense to introduce the following notion of quantization. Given a set $T$, we call the quantum set  $qT:=\Y(\Bbbk T)$, the  {\bf quantization} of a set  $T$. Note that the evaluation of a quantum set on the final coalgebra $\Bbbk \{\bullet\}$ defines a functor in the opposite direction $q\set\rightarrow \set$ which can be ragarded as a retraction of a quantum set to its classical part, since $qT(\Bbbk \{\bullet\})=T$.

 \subsubsection{Birkhoff--von Neumann's quantum propositional calculus}
\begin{corollary}\label{truth}
The quantization $q\Omega$ of the subobject classifier $\Omega=\{ \bm{\bot,} \bm{\top} \}$ (a.k.a. the set of truth values) in the topos $\set$ under the above embedding $\set\hookrightarrow q\set$ is the quantum set of idempotents in the internal quantum function algebra $F=q\Bbbk$ in $q\set$, i.e.
\begin{align}\label{idem}
\Y(\Bbbk \Omega \})(C)=
\big\{ x\in F(C)\ \big|\    x^{2} =x\big\}.
\end{align}
\end{corollary}

\begin{proof} By Proposition \ref{setqset}, the right hand side of (\ref{ortidem}) consists now of two elements $x_{\bot}, x_{\top}\in F(C)$ satisfying
\begin{align}
x_{\bot}^2=x_{\bot},\  x_{\bot}x_{\top}=x_{\top}x_{\bot}=0,\  x_{\top}^2=x_{\top},\  x_{\bot}+x_{\top}=1,
\end{align}
and hence it is uniquely determined by $x:=x_{\top}$ satisfying $x^2=x$.
\end{proof}

Note that $\Omega$ is a Boole algebra. The linearization functor applied to Boolean operations $\wedge$, $\vee$ transforms them into coalgebra maps, and after applying the Yoneda embedding, a partial distributive lattice structure under the induced operations
\begin{align}
\Y(\Bbbk \Omega )\times_{\cat}\Y(\Bbbk \Omega )\rightarrow \Y(\Bbbk \Omega \otimes\Bbbk\Omega )\rightarrow  \Y(\Bbbk (\Omega \times\Omega))\rightarrow  \Y(\Bbbk \Omega).
\end{align}

This realizes the Birkhoff-von Neumann idea of {\bf  quantum propositions} as projections in a von Neumann algebra. In our formalism however, it is not enough to have them commuting before one aplies the binary Boolean operations $\wedge$ and $\vee$, we need them first to form an admissible pair $(x_1, x_2)\in (\Y(\Bbbk \Omega )\times_{\cat}\Y(\Bbbk \Omega ))(C)$ before we multiply them in the von Neumann-like partial commutative internal quantum function algebra $F$.

\subsubsection{Quantum states} The quantum propositional calculus is not enough in physical applications. When one interpretes elements of $F$ as quantum observables, the important question about their expected values in a given state  arises. Ignoring for the time being positivity part of the story, we define states as follows.
\begin{definition}
The  {\bf set of states} ${\rm St}(C)$ of a given coalgebra $C$ is an affine space of elements $c\in C$ satisfying $\varepsilon(c)=1$. We will call the affine space $\langle-\rangle(C)$ of affine functions ${\rm St}(C)\rightarrow \Bbbk$ the {\bf space of expectations} on $C$. It is easy to see that they form a presheaf of affine spaces on $\coalg$, and hence  an internal affine space in $q\set$. Moreover, there is a tautological linear map of internal affine spaces in  $q\set$ (here we suppress the forgetful functor from algebras to affine spaces)
\begin{align}
F\rightarrow \langle-\rangle,
\end{align}
defined as a natural transformation of presheaves, for $x\in F(C)$ given as 
\begin{align}
(C\ni c\mapsto x(c)\in \Bbbk)\mapsto ({\rm St}(C)\ni c\mapsto x(c)\in \Bbbk)
\end{align}
which we call  {\bf expectation} on states of $C$. Its evaluation at a given state  will be called the  {\bf expected value} of a given observable $x\in F(C)$ in the state $c\in {\rm St}(C)$, and denoted by $ \langle x\rangle_{c}$.
\end{definition}

 In particular, an evaluation of an idempotent element of $F(C)$ at a given state  $c\in {\rm St}(C)$ is a  purely algebraic counterpart of the notion of  {\bf probability} of a {\bf quantum proposition} at a given {\bf state} on the von Neumann algebra. In the internal language of  $q\set$ it is a composition
\begin{align}
\Y(\Bbbk \Omega )\longrightarrow F \longrightarrow \langle-\rangle
\end{align}
transforming the {\bf truth value} into its {\bf probability} as follows $\top\mapsto  \langle x_{\top}\rangle_{c}$.

\subsubsection{Positivity}

In modelling the probabilistic quantum-physical measurement we have to restrict the above constructions to the category of $*$-coalgebras over the field of complex numbers $\Bbbk=\C$, i.e. coalgebras equipped with an antilinear, involutive, the comultiplication flipping and the counit preserving  map $(-)^{*}: C\rightarrow C $ as objects  and $*$-preserving coalgebra maps as morphisms. To the previous  purely algebraic condition $\varepsilon(c)=1$ on $c\in C$ to be a state, we have to add the positivity condition $\Delta(c)=\sum_{i}c^{*}_{i}\otimes c_{i}$. Note that such states form a convex subset of the above complex affine space of purely algebraic states, since for every $t\in[0, 1]$
\begin{align}
\begin{split}
 & \Delta((1-t)c_{0}+tc_{1})=(1-t)\Delta(c_{0})+t\Delta(c_{1})
=\  (1-t)\sum_{i}c^{*}_{0, i}\otimes c_{0, i}+t\sum_{i}c^{*}_{1, j}\otimes c_{1, j}\\
=\ &  \sum_{i}(\sqrt{1-t}\ c_{0, i})^{*}\otimes (\sqrt{1-t}\ c_{0, i})+\sum_{j}(\sqrt{t}\ c_{1, j})^{*}\otimes (\sqrt{t}\ c_{1, j}).
\end{split}
\end{align}

 Then the induced antilinear, involutive, the multiplication reversing and the unit preserving map $(-)^{*}: F(C)\rightarrow F(C)$, $x^{*}(c):=\overline{x(c^{*})}$ makes $F(C)$ a $*$-algebra. A state $c$ becomes now a positive functional on $F(C)$, as it follows from the following calculation.
\begin{align}\label{posit}
\begin{split}
&\langle x^{*}x\rangle_{c}=(x^{*}x)(c)=x^{*}(c_{(1)})x(c_{(2)})=\sum_{i}x^{*}(c^{*}_{i})x(c_{i})\\
=\ & \sum_{i}\overline{x(c^{**}_{i})}x(c_{i})=\sum_{i}\overline{x(c_{i})}x(c_{i})=\sum_{i}\big|x(c_{i}\big|^{2}\geq 0.
\end{split}
\end {align}

When on the linearization $\C\{ \bot, \top\}$ we define $*$ in the way that $\bot^*=\bot$, $\top^*=\top$, we obtain $x_{\bot}^{*}=x_{\bot}$,  $x_{\top}^{*}=x_{\top}$. Therefore, by \eqref{posit},
\begin{align}
\begin{split}
&\langle x_{\top}\rangle_{c}=\langle x_{\top}^2\rangle_{c}=\langle x_{\top}^{*}x_{\top}\rangle_{c}\geq 0.
\end{split}
\end {align}
The same holds for $x_{\bot}=1-x_{\top}$, hence by $\langle 1\rangle_{c}=\varepsilon(c)=1$ we have $\langle x_{\top}\rangle_{c}\in [0, 1]$. Therefore $\langle x_{\top}\rangle_{c}$ can be interpreted as probability of the proposition $x_{\top}$ in the state $c$.

    

\subsubsection{Associative algebras as quantum algebras}
The construction of the internal partial commutative algebra $F$ can be regarded as a base quantum algebra obtained by quantizing  the base field $\Bbbk$, , i.e. $F=q \Bbbk$.. Let us explain it and show that it extends naturally to quantization of arbitrary $\Bbbk$-algebras.

\begin{definition}\label{qAlg}
    For every unital associative $\Bbbk$-algebra $A$ its  Sweedler cofree coalgebra ${\rm C}(A)$ \cite{S-69}  represents a  presheaf $qA:=\Y({\rm C}(A))$ on $\coalg$ 
       $${C} \mapsto \coalg(C,{\rm C}(A)) = \vect(C, A)$$
    of algebras (with convolution multiplication), regarded as  free algebras  for a monad $\vect(C, -)$ on the category of $\Bbbk$-algebras \cite{ML-71}.  We call such a presheaf object {\bf quantization} of the algebra $A$. In this way we obtain a presheaf $q\alg$ of categories on $\coalg$ where 
\begin{align}
q\alg\big(qA', qA\big) (C) :=\ &  \alg^{\vect(C, -)}(\vect(C, A'), \vect(C, A)).
\end{align}
\end{definition} 

\begin{remark}Note that the term {\bf deformation quantization} splits then into two stages; first, the quantization of a (typically commutative) algebra in our sense, and then a deformation of the thus obtained quantum algebra.
\end{remark}

 One could also note that, despite we use the same prefix $q$ for quantization of sets and algebras, the quantization of an algebra differs from the quantization of its underlying set. The appropriate relation between quantum algebras and quantum sets is provided by the following proposition where ${\rm Meas}( A', A)$ denotes Sweedler's universal measuring coalgebra \cite{S-69}.

\begin{proposition}
The category $q\alg$   of quantum algebras provides a canonical  enrichment of the category $\alg$ of algebras in the category of quantum sets $q\set$ as follows
$$q\alg\big(qA', qA\big) =  \Y\big({\rm Meas}( A', A)\big).$$
\end{proposition}
\begin{proof}
By Definition \ref{qAlg} and  the existence of the universal measuring coalgebra \cite{S-69} we get
\begin{align}
   & q\alg\big(qA', qA\big) (C) \\
:=\ &  \alg^{\vect(C, -)}(\vect(C, A'), \vect(C, A)) \\
=\ &  \alg( A', \vect(C, A)) \\
=\ &  \coalg\big(C, {\rm Meas}( A', A)\big)\\
=\ &  \Y\big({\rm Meas}( A', A)\big)(C)
\end{align}
where the second equality follows from the equivalence between  the full subcategory of free algebras in the Eilenberg-Moore category and the Kleisli category \cite{K-65}\cite{ML-71}. 
\end{proof}

Lastly, since the same arguments as in the proof of Proposition \ref{partcomm} can be applied to any commutative $\Bbbk$-algebra instead of $\Bbbk$ itself, we get the following proposition.

\begin{proposition}
If an algebra $A$ is commutative, $qA$ is partial commutative, i.e. the quantum partial operation
$$ qA\times_{\cat} qA\longrightarrow qA,$$
induced by multiplication in $A$ and the symmetric monoidal structure of the Yoneda embedding, is commutative.
\end{proposition}

\subsection{Quantum Universal Algebra}
Motivated by the above case of associative algebras, below we will consider representable quantum sets, represented by objects of any algebraic kind compatible with its coalgebra structure. The above partial symmetric monoidal structure of the Yoneda embedding will make them presheaves of partial algebras of that kind.
\begin{definition}
    A {\bf quantum operation} of arity $n \in \N$ on a quantum set represented by a coalgebra $D$ is a coalgebra map 
    $$\omega:D^{\otimes n} \rightarrow D.$$
\end{definition}
Note that quantum operations together with the above partial monoidal structure of the Yoneda embedding induce following partial operations on representable quantum sets, defined as compositions
$$\Y(D)\times_{\cat}\cdots \times_{\cat} \Y(D)\longrightarrow \Y(D\otimes \cdots \otimes D) \longrightarrow\Y(D).$$
\begin{definition}
    A {\bf quantum identity} is a commutative diagram of morphisms of representable quantum sets induced by quantum operations and structural morphism of the symmetric monoidal category $\coalg$.  
\end{definition}

\begin{example}[Binary operations]
Every {\bf binary operation} $$\circ : B\otimes B \rightarrow B,\quad  b \otimes b' \mapsto b \circ b'$$
allows us to define a {\bf partial binary operation} on $\Y (B)$ as a partially defined composition 
$$(\Y (B)\times_{\cat} \Y (B))(C) \rightarrow \Y(B\otimes B)(C) \rightarrow \Y(B)(C)$$
$$(\beta, \beta') \mapsto \big(c \mapsto \beta(c_{(1)}) \otimes \beta'(c_{(2)})\big) \mapsto  \beta\circ \beta':= \big(c \mapsto \beta(c_{(1)}) \circ \beta'(c_{(2)})\big).$$
\end{example}

\subsubsection{Quantum Boolean Algebras}
The aim of extending the notion of Boolean algebra to its quantum counterpart is to provide an algebraic semantics \cite{Bl-Pi-89} for quantum propositional calculus beyond the Birkhoff--von Neumann propositional calculus, by allowing quantum truth values. Postponing for a while the question of the negation in a  Boolean algebra, we start with bounded distributive lattices.

\begin{definition}
A representable  quantum {\bf bounded distributive lattice} is a quantum set represented by a coalgebra $B$  equipped with coalgebra maps
$$\Bbbk\{\bullet\} \stackrel{\bm{\bot}}{\longrightarrow} B,
\qquad \Bbbk\{\bullet\} \stackrel{\bm{\top}}{\longrightarrow} B,
\qquad B \otimes_\Bbbk B \stackrel{\bm{\wedge}}{\longrightarrow} B,
\qquad B \otimes_\Bbbk B \stackrel{\bm{\vee}}{\longrightarrow} B$$
satisfying commutativity of the following diagrams in $\coalg$, where $\tau$ denotes the transposition of tensor factors.

\emph{Associativity}
\begin{center}
\begin{tikzcd}
B \otimes_\Bbbk B\otimes_\Bbbk B \arrow[dd, "\bm\wedge\otimes_\Bbbk B"'] \arrow[rr, "B \otimes_\Bbbk \bm\wedge"] &  & B \otimes_\Bbbk B \arrow[dd, "\bm\wedge"] \\
&  &\\
B\otimes_\Bbbk B \arrow[rr, "\bm\wedge"]&  & B                                      
\end{tikzcd}\hspace{2em}
\begin{tikzcd}
B \otimes_\Bbbk B\otimes_\Bbbk B \arrow[dd, "\bm\vee\otimes_\Bbbk B"'] \arrow[rr, "B \otimes_\Bbbk \bm\vee"] &  & B \otimes_\Bbbk B \arrow[dd, "\bm\vee"] \\
&  &\\
B\otimes_\Bbbk B \arrow[rr, "\bm\vee"]&  & B                                      
\end{tikzcd}
\end{center}

\emph{Identity}
\begin{center}
\hspace{4.5em}
\begin{tikzcd}
B \otimes_\Bbbk B \arrow[r, "\bm\wedge"]& B \\ &  &   \\
B \otimes_\Bbbk \Bbbk \{\bullet\} \arrow[ruu,"\bm\cong"'] \arrow[uu, "B \otimes_\Bbbk \bm\top"] &
\end{tikzcd}\hspace{6em}
\begin{tikzcd}
B \otimes_\Bbbk B \arrow[r, "\bm\vee"]& B \\ &  &   \\
B \otimes_\Bbbk \Bbbk \{\bullet\} \arrow[ruu,"\bm\cong"'] \arrow[uu, "B \otimes_\Bbbk \bm\bot"] &
\end{tikzcd}
\end{center}

\emph{Commutativity}

\begin{center}
\begin{tikzcd}
B\otimes_\Bbbk B \arrow[rr, "{\tau}"] \arrow[rd, "\bm\vee"'] &   & B\otimes_\Bbbk B \arrow[ld, "\bm\vee"] \\
& B &
\end{tikzcd} \hspace{4em}
\begin{tikzcd}
B\otimes_\Bbbk B \arrow[rr, "{\tau}"] \arrow[rd, "\bm\wedge"'] &   & B\otimes_\Bbbk B \arrow[ld, "\bm\wedge"] \\
& B &
\end{tikzcd}
\end{center}

\newpage
\emph{Absorbtion}

\hspace{2.88em}
\begin{tikzcd}
B \otimes_\Bbbk B \otimes_\Bbbk B \arrow[r, "B \otimes_\Bbbk \bm\vee"]& B \otimes_\Bbbk B \arrow[dd, "\bm\wedge"] \\
&\\
B \otimes_\Bbbk B\arrow[uu, "\Delta \otimes_\Bbbk B"] \arrow[r, "B\otimes_\Bbbk \varepsilon"] & B
\end{tikzcd}\hspace{4em}
\begin{tikzcd}
B \otimes_\Bbbk B \otimes_\Bbbk B \arrow[r, "B \otimes_\Bbbk \bm\wedge"]& B \otimes_\Bbbk B \arrow[dd, "\bm\vee"] \\
&\\
B \otimes_\Bbbk B\arrow[uu, "\Delta \otimes_\Bbbk B"] \arrow[r, "B\otimes_\Bbbk \varepsilon"] & B
\end{tikzcd}

\emph{Distributivity}

\begin{center}
\hspace{-3em}
\begin{tikzcd}
B\otimes_\Bbbk B\otimes_\Bbbk B \arrow[rr, "B \otimes _\Bbbk \bm \wedge"] \arrow[dd, "\Delta  \otimes _\Bbbk B  \otimes _\Bbbk B"'] &  & B \otimes _\Bbbk B \arrow[dddddd, "\bm\vee"] \\
&  &\\
B\otimes_\Bbbk  B\otimes_\Bbbk B\otimes_\Bbbk B \arrow[dd, "B  \otimes _\Bbbk \tau  \otimes _\Bbbk B"']&  &\\
&  &\\
B\otimes_\Bbbk B\otimes_\Bbbk B \otimes_\Bbbk B \arrow[dd, "\bm\vee \otimes_\Bbbk \bm \vee"']&  &\\
&  &\\
B\otimes_\Bbbk B \arrow[rr, "\bm \wedge"]&  & B
\end{tikzcd}
\begin{tikzcd}
B\otimes_\Bbbk B\otimes_\Bbbk B \arrow[rr, "B \otimes _\Bbbk \bm \vee"] \arrow[dd, "\Delta  \otimes _\Bbbk B  \otimes _\Bbbk B"'] &  & B \otimes _\Bbbk B \arrow[dddddd, "\bm\wedge"] \\
&  &\\
B\otimes_\Bbbk B\otimes_\Bbbk B\otimes_\Bbbk B \arrow[dd, "B  \otimes _\Bbbk \tau  \otimes _\Bbbk B"']&  &\\
&  &\\
B\otimes_\Bbbk B\otimes_\Bbbk B\otimes_\Bbbk B \arrow[dd, "\bm\wedge \otimes_\Bbbk \bm \wedge"']&  &\\
&  &\\
B\otimes_\Bbbk B \arrow[rr, "\bm \vee"]&  & B
\end{tikzcd}
\end{center}
$$b\vee(b'\wedge b'') = (b_{(1)} \vee b') \wedge( b_{(2)} \vee b''), \quad 
\qquad b\wedge(b'\vee b'') = (b_{(1)} \wedge b') \vee( b_{(2)} \wedge b'')$$

\begin{definition}
We say, that a representable quantum bounded distributive lattice is a representable {\bf quantum Boolean algebra} if there exists an involutive coalgebra map $\bm\neg: B\rightarrow B^{o}$ making the following diagrams in $\vect$ commute. 
\end{definition}
\begin{center}
\hspace{-0.5em}
\begin{tikzcd}
B \otimes_\Bbbk B \arrow[rr, "B \otimes_\Bbbk \bm\neg"]&& B \otimes_\Bbbk B \arrow[dd, "\bm\vee"]   \\
&&\\
B \arrow[uu, "\Delta"] \arrow[r, "\varepsilon"] \arrow[dd, "\Delta"'] & \Bbbk \{ \bullet \} \arrow[r, "\bm \top"] & B\\
&&\\
B \otimes_\Bbbk B \arrow[rr, "\bm\neg \otimes _\Bbbk B"']&& B \otimes_\Bbbk B \arrow[uu, "\bm \vee"']
\end{tikzcd}\hspace{3.75em}
\begin{tikzcd}
B \otimes_\Bbbk B \arrow[rr, "B \otimes_\Bbbk \bm\neg"]&& B \otimes_\Bbbk B \arrow[dd, "\bm\wedge"]   \\
&&\\
B \arrow[uu, "\Delta"] \arrow[r, "\varepsilon"] \arrow[dd, "\Delta"'] & \Bbbk \{ \bullet \} \arrow[r, "\bm \bot"] & B\\
&&\\
B \otimes_\Bbbk B \arrow[rr, "\bm\neg \otimes _\Bbbk B"']&& B \otimes_\Bbbk B \arrow[uu, "\bm \wedge"']
\end{tikzcd}
\end{center}
 We will refer to it as \emph{Complement} axiom. It reads as the following system of identities
$$ (\neg(b_{(1)})) \vee b_{(2)}
= \varepsilon(b) \top, \qquad (\neg(b_{(1)})) \wedge b_{(2)} = \varepsilon(b) \bot ,$$

$$ b_{(1)} \vee\neg(b_{(2)})= \varepsilon(b) \top, \qquad b_{(1)} \wedge \neg(b_{(2)}) = \varepsilon(b) \bot .$$
where for simplicity of notation we use the following conventions:
\begin{align*}
b  \vee b' &:= \bm \vee (b \otimes_\Bbbk b')  &b  \wedge b' &:= \bm \wedge (b \otimes_\Bbbk b')  \\
\bot &:= \bm\bot(\bullet) &\top &:= \bm\top(\bullet)
\end{align*}
\end{definition}

The following proposition is the sanity check for these axioms, saying that they restrict to classical ones for classical Boolean algebras.

\begin{proposition}
The linearization $\Bbbk B$  of a Boolean algebra $B$ represents a quantum Boolean algebra.
\end{proposition}
\begin{proof}
    Let $(B, \bot, \top, \neg, \wedge, \vee)$ be a Boolean algebra and $\Bbbk B$ its linearization.

    It is enough to check the axioms on the set $B$ of generators of its linearization $\Bbbk B$, and the rest will follow from linearity.

\noindent\emph{Identity.}     Since $\Bbbk \{\bullet\}$ is a monoidal unit we have have the natuaral isomorphism
                $$\Bbbk B \otimes _\Bbbk \Bbbk \{\bullet\} \stackrel{ \cong }{\longrightarrow} \Bbbk B$$
                $$b \otimes_\Bbbk \bullet \mapsto b$$
    On the other hand we have
    $${B \otimes_\Bbbk \bm\bot}: b \otimes_\Bbbk \bullet \mapsto b \otimes_\Bbbk \bm \bot.$$
    But from the definition of $\bm\vee$ and $\bm \bot$ we get
    
    $$\bm\vee (b \otimes_\Bbbk \bm \bot) = b \vee \bot = b$$
    so the identity diagram commutes.

\noindent\emph{Commutativity.}   Let us chase the commutativity diagram:

        \begin{center}
\begin{tikzcd}
b \otimes_\Bbbk b' \arrow[rrr, maps to, "\tau"] \arrow[rd,maps to, "\bm\vee"'] &&& b' \otimes_\Bbbk b \arrow[ld,maps to, "\bm\vee"] \\
& \bm\vee(b \otimes_\Bbbk b') & \bm\vee(b' \otimes_\Bbbk b) &          
\end{tikzcd}    
\end{center}
    But we know that

    $$ \bm\vee(b \otimes_\Bbbk b') = b\vee b' = b' \vee b = \bm\vee(b' \otimes_\Bbbk b),$$
    so the diagram commutes.

\noindent\emph{Distributivity.}  Again, we chase the relevant diagram:
\begin{center}
    \begin{tikzcd}
b\otimes_\Bbbk b'\otimes_\Bbbk b'' \arrow[rr,maps to, "B \otimes _\Bbbk \bm \wedge"] \arrow[dd,maps to, "\Delta  \otimes _\Bbbk B  \otimes _\Bbbk B"'] &  & {b \otimes _\Bbbk \bm\wedge(b', b'')} \arrow[ddddd, maps to,"\bm\vee"] \\
&  &\\
b\otimes_\Bbbk b\otimes_\Bbbk b'\otimes_\Bbbk b'' \arrow[dd,maps to, "B  \otimes _\Bbbk \tau  \otimes _\Bbbk B"']   &  &\\
&  &\\
b\otimes_\Bbbk b'\otimes_\Bbbk b\otimes_\Bbbk b'' \arrow[dd, "\bm\vee \otimes_\Bbbk \bm \vee"']&  &\\
&  & {\bm\vee(b,\bm\wedge(b', b''))}\\
{\bm\vee(b, b')\otimes_\Bbbk \bm\vee(b, b'')} \arrow[rr,maps to, "\bm \wedge"]&  & {\bm\wedge(\bm\vee(b, b'), \bm\vee(b, b''))}.
\end{tikzcd}  
\end{center}
    The distributive law of the classical Boolean algebra $B$ implies that 

    $$ \bm\vee(b,\bm\wedge(b', b'')) = b \vee (b' \wedge b'') = 
    (b \vee b') \wedge(b\vee b'') 
    = \bm\wedge(\bm\vee(b, b')  ,\bm\vee(b, b'') )$$
    so the diagram commutes.
    Commutativity of the second diagram follows similarly. 

\noindent\emph{Complements.}: Finally, let us look at complements:
    Thanks to the commutativity axiom, in the first diagram it is enough to check only the upper half.
\begin{center}
\begin{tikzcd}
{} \arrow[phantom, loop, distance=2em, in=55, out=125]           &&&&\\
&&&&\\
b \otimes_\Bbbk b \arrow[rrrr, "B \otimes_\Bbbk \bm\neg"]             &&&& b \otimes_\Bbbk \bm\neg(b) \arrow[d, "\bm\vee", maps to] \\
&&&& \bm\vee(b \otimes_\Bbbk \bm\neg(b))                      \\
b \arrow[uu, "\Delta", maps to] \arrow[rr, "\varepsilon", maps to] && \bullet \arrow[rr, "\bm\top", maps to] &  & \top
\end{tikzcd}
\end{center}
From the complement axiom for the Boolean algebra $B$ it follows that 
    $$\bm\vee(b \otimes_\Bbbk \bm\neg(b)) = b \vee \neg b =\top,$$
    hence this upper half does commute. Commutativity of the second diagram will follow in a similar fashion.
\end{proof}
Better still, many properties, when appropriately generalized, of classical Boolean algebras also apply to  representable quantum Boolean algebras. 
\begin{theorem}
    Negation in a representable quantum bounded distributive lattice  is unique.
\end{theorem}

\begin{proof}
Assume we have two negation coalgebra maps $ {\bm\neg}, \widetilde{\bm\neg} : B \rightarrow B^{op}$.
Using the following axioms of a representable quantum bounded distributive lattice and the complement axiom for the negation morphism we get

    \begin{equation}
        \begin{split}
          &  \widetilde{\bm\neg} (b) \\
(Identity )\quad =     \  &  \widetilde{\bm\neg}(b) \wedge \top \\
   (counit) \quad =     \    &  \widetilde{\bm\neg} (b_{(1)} \varepsilon(b_{(2)})) \wedge \top \\
(linearity) \quad =          \  &  \widetilde{\bm\neg} (b_{(1)}) \wedge \varepsilon(b_{(2)})\top \\
(complement)  \quad =       \    &  \widetilde{\bm\neg} (b_{(1)}) \wedge (b_{(2)(1)} \vee \bm\neg(b_{(2)(2)})) \\
(distributivity)  \quad =       \    &  (\widetilde{\bm\neg}(b_{(1)})_{(1)} \wedge b_{(2)(1)})
                \vee
                 (\widetilde{\bm\neg}(b_{(1)})_{(2)} \wedge \bm\neg(b_{(2)(2)})) \\ 
(anticoalg) \quad =         \   &  (\widetilde{\bm\neg}(b_{(1)(2)}) \wedge b_{(2)(1)})
            \vee (\widetilde{\bm\neg}(b_{(1)(1)}) \wedge \bm\neg(b_{(2)(2)}))) \\ 
(renumerate) \quad =         \   &  (
             (\widetilde{\bm\neg}(b_{(2)}) \wedge b_{(3)})
            \vee (\widetilde{\bm\neg}(b_{(1)}) \wedge \bm\neg(b_{(4)}))) \\ 
 (renumerate) \quad =       \    & 
             (
                \widetilde{\bm\neg}(b_{(2)(1)}) \wedge b_{(2)(2)})
            \vee
            (
            \widetilde{\bm\neg}(b_{(1)}) \wedge \bm\neg(b_{(3)})) \\ 
(complement) \quad =        \    &  (
            \varepsilon(b_{(2)})\bot
            \vee
             (
            \widetilde{\bm\neg}(b_{(1)}) \wedge \bm\neg(b_{(3)}))) \\
  (complement)  \quad =      \   & (
             (
                b_{(2)(1)} \wedge \bm\neg(b_{(2)(2)}))
            \vee
             (
            \widetilde{\bm\neg}(b_{(1)}) \wedge \bm\neg(b_{(3)})))  \\
(renumerate) \quad =       \     &  (
             (
                b_{(2)} \wedge \bm\neg(b_{(3)(1)}))
            \vee
            (
            \widetilde{\bm\neg}(b_{(1)}) \wedge \bm\neg(b_{(3)(2)})))  \\
(anticoalg)  \quad =        \   &   (
             (
                b_{(2)} \wedge \bm\neg(b_{(3)})_{(2)}))
            \vee
            (
            \widetilde{\bm\neg}(b_{(1)}) \wedge \bm\neg(b_{(3)})_{(1)}))  \\
 (commutativity)\quad =         \   &  (
             (
            \widetilde{\bm\neg}(b_{(1)}) \wedge \bm\neg(b_{(3)})_{(1)})
            \vee
             (
            b_{(2)}) \wedge \bm\neg(b_{(3)})_{(2)}))  \\
(commutativity) \quad =       \     &  (
             (
                \bm\neg(b_{(3)})_{(1)}) \wedge \widetilde{\bm\neg}(b_{(1)}) 
            \vee
            \bm\wedge (
            \bm\neg(b_{(3)})_{(2)} \wedge b_{(2)} ))  \\
(distributivity)\quad =        \     & 
                ( \bm\neg(b_{(3)}) \wedge
                 (\widetilde{\neg} (b_{(1)}) \vee b_{(2)})) \\
(renumerate) \quad =        \    & 
                ( \bm\neg(b_{(2)}) \wedge
                 (\widetilde{\neg} (b_{(1)(1)}) \vee b_{(1)(2)})) \\
(complement) \quad =       \     & 
                (\bm\neg(b_{(2)}) \wedge
                \varepsilon(b_{(1)})\top)\\
(linearity) \quad =        \    & 
                (\bm\neg(\varepsilon(b_{(1)}) b_{(2)}) \wedge
                \top) \\
 (counit)  \quad =     \   &
                (\bm\neg(b) \wedge
                \top) \\
(identity) \quad        =\  &      \bm\neg(b)
        \end{split}
    \end{equation}
\end{proof}

In the classical Boolean setting, by the complement axiom, the first de Morgan law$$ \bm\neg (b \vee b') 
    = \bm\neg( b) \wedge \bm\neg (b'),$$ implies its apparently weaker version
    \begin{equation}
        \begin{split}
     (b \vee b') \vee (\neg(b) \wedge \neg(b'))=(b \vee b') \vee (\neg(b \vee b'))=\top\\
     (b \vee b') \wedge (\neg(b) \wedge \neg(b'))=(b \vee b') \wedge (\neg(b \vee b'))=\bot.
     \end{split}
    \end{equation}
    However, by uniquness of the complement argument, these two versions are equivalent. Similarly, the second de Morgan law 
    $$ \bm\neg (b \wedge b') 
    = \bm\neg( b) \vee \bm\neg (b')$$
    is equivalent to its weak version
    \begin{equation}
        \begin{split}
     (b \wedge b') \wedge (\neg(b) \vee \neg(b'))=(b \wedge b') \wedge (\neg(b \wedge b'))=\bot\\
     (b \wedge b') \vee (\neg(b) \vee \neg(b'))=(b \wedge b') \vee (\neg(b \wedge b'))=\top.
     \end{split}
    \end{equation}
    
    In the next theorem we prove a quantum counterpart of the weak de Morgan laws.
\begin{theorem}
In every representable quantum Boolean algebra the following weak de Morgan laws hold
\begin{itemize}
    \item 1st de Morgan law 
    $$(b_{(1)}  \vee b'_{(1)}) \vee (\neg(b_{(2)}) \wedge \neg(b'_{(2)}))= \varepsilon(b)\varepsilon(b')\top,
    $$ 
    $$(b_{(1)}  \vee b'_{(1)}) \wedge (\neg(b_{(2)}) \wedge \neg(b'_{(2)}))= \varepsilon(b)\varepsilon(b')\bot.$$ 
    \item 2nd de Morgan law 
    $$(b_{(1)}  \wedge b'_{(1)}) \wedge (\neg(b_{(2)}) \vee \neg(b'_{(2)}))= \varepsilon(b)\varepsilon(b')\top,
    $$ 
    $$(b_{(1)}  \wedge b'_{(1)}) \vee (\neg(b_{(2)}) \vee \neg(b'_{(2)}))= \varepsilon(b)\varepsilon(b')\bot.$$ 
\end{itemize}

\end{theorem}
\begin{proof}
Using the axioms listed below we proceed as follows.

    \begin{equation}
        \begin{split}
        &  (b_{(1)}  \vee b'_{(1)}) \vee (\neg(b_{(2)}) \wedge \neg(b'_{(2)}))  \\
        (distr)\quad  =\ &((b_{(1)}  \vee b'_{(1)})_{(1)} \vee \neg(b_{(2)} )
        \wedge ((b_{(1)}  \vee b'_{(1)})_{(2)} \vee \neg(b'_{(2)} ))\\
        (coalg)\quad  =\ &((b_{(1)(1)}  \vee b'_{(1)(1)}) \vee \neg(b_{(2)} ))
        \wedge ((b_{(1)(2)}  \vee b'_{(1)(2)}) \vee \neg(b'_{(2)} ))\\
        (renumerate)\quad  =\ &((b_{(1)}  \vee b'_{(1)}) \vee \neg(b_{(3)} ))
        \wedge ((b_{(2)}  \vee b'_{(2)}) \vee \neg(b'_{(3)} ))\\
        (assoc\ \&\ comm)\quad  =\ &(b'_{(1)} \vee (b_{(1)}  \vee \neg(b_{(3)} )))
        \wedge (b_{(2)}  \vee (b'_{(2)} \vee \neg(b'_{(3)}))) \\
        (complement)\quad  =\ &(b'_{(1)} \vee (b_{(1)}  \vee \neg(b_{(3)} )))
        \wedge (b_{(2)}  \vee \varepsilon(b'_{(2)})\top) \\
        (linearity)\quad  =\ &(b'_{(1)}\varepsilon(b'_{(2)}) \vee (b_{(1)}  \vee \neg(b_{(3)} )))\wedge (b_{(2)} \vee \top) \\
        (counit)\quad =\ &(b' \vee (b_{(1)}  \vee \neg(b_{(3)} )))\wedge (b_{(2)} \vee \top) \\
        (\top \ absorbs)\quad  =\ &(b' \vee (b_{(1)}  \vee \neg(b_{(3)} )))\wedge (\varepsilon(b_{(2)})\top) \\
        (linearity)\quad  =\ &(b' \vee (b_{(1)}  \vee \neg(\varepsilon(b_{(2)})b_{(3)} )))\wedge \top \\
        (counit)\quad  =\ &(b' \vee (b_{(1)}  \vee \neg((b_{(2)} )))\wedge \top \\
        (identity)\quad  =\ &b' \vee (b_{(1)}  \vee \neg(b_{(2)} )) \\
        (complement)\quad  =\ &b' \vee \varepsilon(b)\top\\
        (linearity\ \&\ \top \ absorbs)\quad =\ & \varepsilon(b)\varepsilon(b')\top.
        \end{split}
    \end{equation}

    \begin{equation}
        \begin{split}
         & (b_{(1)}  \vee b'_{(1)}) \wedge (\neg(b_{(2)}) \wedge \neg(b'_{(2)}))   \\
        (distr)\quad   =\ &(b_{(1)} \wedge (\neg(b_{(2)}) \wedge \neg(b'_{(2)}))_{(1)}) \vee (b'_{(1)} \wedge (\neg(b_{(2)}) \wedge \neg(b'_{(2)}))_{(2)})  \\
        (coalg)\quad  =\ &(b_{(1)}  \wedge (\neg(b_{(2)})_{(1)} \wedge \neg(b'_{(2)})_{(1)} ))
        \vee (b'_{(1)}  \wedge (\neg(b_{(2)})_{(2)} \wedge \neg(b'_{(2)})_{(2))}))\\
        (\neg\text{ is anticoalg})\quad   =\ &(b_{(1)}  \wedge (\neg(b_{(2)(2)}) \wedge \neg(b'_{(2)(2)} )))
        \vee (b'_{(1)}  \wedge (\neg(b_{(2)(1)} \wedge \neg(b'_{(2)(1))}))\\
        (\text{ren, ass, com})\quad   =\ &(b_{(1)}  \wedge \neg(b_{(2)(2)})) \wedge \neg(b'_{(2)} ))
        \vee ((b'_{(1)(1)}  \wedge \neg(b'_{(1)(2)}) \wedge \neg(b_{(2)(1))}))\\
        (complement)\quad   =\ &(b_{(1)}  \wedge \neg(b_{(2)(2)})) \wedge \neg(b'_{(2)(2)} ))
        \vee (\varepsilon(b'_{(1)})\bot \wedge \neg(b_{(2)(1))}))\\
        (\bot \text{absorbs, linearity})\quad   =\ &(
        b_{(1)}  \wedge \neg(b_{(2)(2)})) \wedge \neg(\varepsilon(b'_{(1)})b'_{(2)} ))
        \vee (\varepsilon(\neg(b_{(2)(1))})\bot ))\\
        (\text{counit})\quad   =\ &(
        b_{(1)}  \wedge \neg(b_{(2)(2)})) \wedge \neg(b'))
        \vee (\varepsilon(\neg(b_{(2)(1))})\bot ))\\
        (\text{linearity})\quad   =\ &(
        b_{(1)}  \wedge \neg(\varepsilon((b_{(2)(1))})b_{(2)(2)})) \wedge \neg(b'))
        \vee \bot \\
        (counit)\quad   =\ &(
        b_{(1)}  \wedge \neg(b_{(2)})) \wedge \neg(b'))
        \vee \bot \\
        (complement)\quad   =\ &(
        \varepsilon(b) \bot \wedge \neg(b'))
        \vee \bot \\
        (\bot \text{ absorbs})\quad   =\ &\varepsilon(b) \varepsilon(b') \bot
        \end{split}
    \end{equation}

    For the weak second de Morgan law the proof is similar.

\end{proof}

\section{Quantum Quivers}
Besides single coalgebras, one can consider diagrams of coalgebras, among which the most important are {\bf quantum quivers}.
Given a coalgebra $D$, the (co)opposite  coalgebra is denoted by $D^{o}$.  The following definition is adapted from the definition of the underlying quiver of a quantum category explicitly defined by Chikhladze \cite{C-11}  inspired by Day-Street \cite{D-S-03},  to our context of representable quantum sets.
 \begin{definition}
Consider a pair of coalgebra maps $\partial_0: D_1\rightarrow D_0^{o}$,  $\partial_1: D_1\rightarrow D_0$ satisfying for every $d_1\in D_1$ 
$$\partial_0(d_{1(1)})\otimes \partial_1(d_{1(2)} )= \partial_0(d_{1(2)})\otimes \partial_0(d_{1(1)}).$$ 
Denote the representable presheaves $E:=\Y(D_1)$, $V:=\Y(D_0)$, $V^{o}:=\Y(D_0^{o})$ and call them the representable quantum sets of {\bf edges},   {\bf vertices} and   {\bf opposite vertices} of the representable {\bf quantum quiver}, respectively. 
\end{definition}
Note that  for every $e\in E(C)$ we have $s_C(e):=\partial_0\circ e \in  V^{o}(C)$ and $t_C(e):=\partial_1\circ e \in  V(C)$, which we call the {\bf source} and the {\bf target} of an edge $e$, respectively. Then,  since the full subcategory $\cat$ of $q\set$ whose objects are representable presheaves  is partial monoidal in the sense of Definition \ref{def(1, 2)}, we have
\begin{align}\label{qquiver}
(s_C(e), t_C(e))\in (V\times_{\cat}V^{o})(C).
\end{align}
Now it is clear how to define {\bf quantum quiver}, which is not necessarily representable. First, we fix  some partial-monoidal subcategory $\cat$ of $q\set$ as a super-structure, next, given a quantum set of vertices, we define the quantum set of opposite vertices as follows $V^{o}(C):=V(C^{o})$, and finally we complement it with natural transformations $s: E\rightarrow V^{o}$ and $t: E\rightarrow V$ satisfying (\ref{qquiver}) for every $e\in E(C)$.

\subsection{Correspondences}
\begin{definition}
With any representable quantum set  $X=\Y(C)$ as above we associate a category $\cat_{X}:=\comod^{op}_{\hspace{-0.1em}C}$, the opposite category of right $C$-comodules. With every morphism $f: X\rightarrow Y$,  $Y=\Y(D)$ of representable quantum sets represented by a coalgebra map $C\rightarrow D$ we associate a   pair of adjoint functors  
\begin{center}
\begin{tikzcd}[column sep=-5em, row sep=0.75em]
& \cat_{X} \arrow[bend right=55]{dd}{\hspace{0.2em}f^{*}} & \\
  & \dashv   &  \\                     
 & \cat_{Y}.\arrow[bend right=55]{uu}{f_{*}} &  
\end{tikzcd}
\end{center}
described as follows. The reader should beware that we work here with opposite categories of comodules but we write the structural maps in the category of vector spaces without reversing arrows. The right adjoint {\bf corestriction} functor $f_{*}$ is defined by the composition
$$f_{*}(M\rightarrow M\otimes C):= (M\rightarrow M\otimes C\rightarrow M\otimes D)$$
where the second composed arrow on the right hand side is induced by the coalgebra map $C\rightarrow D$. The  left adjoint {\bf coinduction} functor $f^{*}$ is defined as 
$$f^{*}(N\rightarrow N\otimes D):=(N\Box_{D}C\rightarrow N\Box_{D}C\otimes C)$$
where the  arrow on the right hand side is induced by the comultiplication $C\rightarrow C\otimes C$.
\end{definition}
\begin{definition}\label{discr}
We say that $f$ is {\bf discrete} if there is an adjunction $f_{!}\dashv f^{*}$.
\end{definition}
Definition \ref{discr} is motivated by the following classical example.
\begin{example}
Let $C:=\Bbbk S$ for a set $S$. Then the category $\cat_{X}$ is the opposite category of vector spaces $M$ with an $S$-decomposition $M=\oplus_{s\in S}M_{s}$, with the $C$-comodule structure $M\longrightarrow M\otimes C$ of the form $m_{s}\mapsto m_{s}\otimes s$ for every $m_{s}\in M_{s}$, and morphisms being linear maps preserving the $S$-decomposition. If $f: S\rightarrow T$ is a map of sets, it is equivalent to the coalgebra map $f: C\rightarrow D$ (therefore we use the same $f$ for denoting both) where $D=\Bbbk T$, inducing the following adjunctions $f_{!}\hspace{-0.15em}\dashv f^{*}\hspace{-0.15em}\dashv  f_{*}$ for functors  $f_{!}, f_{*}:\cat_{X}\rightarrow \cat_{Y}$, $ f_{*}:\cat_{Y}\rightarrow \cat_{X}$
\begin{align*}
(f_{!}M)_{t} = \prod_{s\in f^{-1}(t)}M_{s},\ \  
(f^{*}N)_{s} = N_{f(s)},\ \ 
(f_{*}M)_{t} = \bigoplus_{s\in f^{-1}(t)}M_{s}.
\end{align*}
Note that when fibers of $f$ are finite (such maps are called {\bf quasi-finite}), $f_{!}=f_{*}$ hence $f_{*}$ and $f^{*}$ form a {\bf Frobenius pair} of functors \cite{C-M-Z-07}. 
\end{example}
Here is an example where we allow our sets to be quantum, representable by coalgebras. 
\begin{example}\label{quant-corr}
Every coalgebra map $f: C\rightarrow D$ induces an adjunction $f^{*}\dashv f_{*}$ of functors $f_{*}:\cat_{X}\rightarrow \cat_{Y}$, $ f^{*}:\cat_{Y}\rightarrow \cat_{X}$ as above. If  $T$  is a $(C, D)$-bicomodule such that there is an isomorphism of functors $cohom_{D}(T, -)\xlongrightarrow{\cong} -\ \Box_{D}C$, by the definition of $cohom$  as a left adjoint functor between comodule categories\cite{T-77}, the left adjoint  $f_{!}:\cat_{X}\rightarrow \cat_{Y}$, $f_{!}M=M\Box_{C}T$, to $f^{*}$  exists implying that $f$ is discrete. If, in addition, $T$  is  isomorphic to $C$ as a $(C, D)$-bicomodule, we have $f_{!}=f_{*}$, i.e.  $f_{*}$ and  $f^{*}$ form a Frobenius pair of functors.
\end{example}

\begin{definition}
By a {\bf quantum correspondence} from $X$ to $Y$ we mean a morphism  of representable quantum sets of the form $F\longrightarrow  Y\times_{\cat}  X^{o}$. The associated diagram in $q\set$
\begin{center}
\begin{tikzcd}[column sep=small]
& F \arrow[dl, "t", swap] \arrow[dr, "s"] & \\
 Y & &  X^{o}
\end{tikzcd}
\end{center}
defines the {\bf source}  and the {\bf target} morphisms, $s$ and $t$, respectively. We say that the correspondence is a quantum {\bf multivaled map} from $X$ to $Y$, if $s$ is discrete. Any quantum multivaled map  from $X$ to $Y$ between representable quantum sets represented by coalgebras $D_{X}$ and $D_{Y}$, respectively,  with $F$ represented by a coalgebra $D_{F}$, defines adjunctions 
$s_{!}\dashv s^{*}\dashv s_{*}$ and $t^{*}\dashv t_{*}$ where  $s_{!}, s_{*}: \cat_{F}\rightarrow \cat_{X}$,  $s^{*}: \cat_{X}\rightarrow \cat_{F}$ ,  $t^{*}: \cat_{Y}\rightarrow \cat_{F}$ and $t_{*}: \cat_{F}\rightarrow \cat_{Y}$  as in Example \ref{quant-corr} with $T$ being $D_{F}$ regarded at first as a right $D_{Y}\otimes D_{X}^{o}$-comodule via the coalgebra map $D_{F}\rightarrow D_{Y}\otimes D_{X}^{o}$, and next  as a $(D_{X}, D_{Y})$-bicomodule by regarding the right $D_{X}^{o}$-coaction as a left $D_{X}$-coaction.
\end{definition} 
\subsection{A categorification of the Cuntz-Pimsner algebra}
Let $Q$ be an adjunction $Q^{*}\dashv Q_{*}$ between two endofunctors on some category $\cat$, with the unit $\eta: {\rm Id}\Rightarrow Q_{*}Q^{*}$ and counit $\varepsilon: Q^{*}Q_{*}\Rightarrow  {\rm Id}$ natural transformations.
\begin{definition}
We define the  {\bf stable category} $\st_{\hspace{-0.15em}Q}$ of the adjunction $Q$ as follows. The objects of $\st_{\hspace{-0.15em}Q}$  are objects $N$ of $\cat$ being equipped  with a retraction of $Q_{*}N$ onto $N$ inducing a retraction of $Q^{*}Q_{*}N$ onto $Q^{*}N$ leaving the counit of the adjunction  stable. Morphisms of  $\st_{\hspace{-0.15em}Q}$  are morphisms of $\cat$ preserving the retraction. 
\end{definition}

The  definition of making an object $N$ belong to  the stable category of the adjunction $Q$ can be rewritten as the existence of two morphisms $
\omega_{N}: N\rightarrow Q_{*}N$ and an opposite-directed morphism  $\sigma_{N}: Q_{*}N\xdashrightarrow N$ (by using the dashed arrow we stress its {\bf wrong-way map} nature)
satisfying the equations
\begin{align}\label{stab-eqns}
\sigma_{N}\circ \omega_{N} =  {\rm Id}_{N}, \hspace{1em}\varepsilon_{N}\circ Q^{*}(\omega_{N}\circ \sigma_{N})  =\varepsilon_{N},
\end{align} 
while the retraction preserving morphisms $\varphi: N\rightarrow N'$  are those satisfying
\begin{align}\label{stab-mor-eqns}
Q_{*}(\varphi)\circ \omega_{N} =  \omega_{N'}\circ  \varphi, \hspace{1em}\varphi\circ \sigma_{N}  =\sigma_{N'}\circ Q_{*}(\varphi).
\end{align}

On the other hand, given an adjunction $Q$ as above, we can form the following monad according to the formalism of \cite{S-72}. 
\begin{definition}
We define the {\bf Cuntz-Pimsner monad} of the adjunction $Q$ as the quotient of the free monad generated by the coproduct $Q^{*}\sqcup Q_{*}$ of endofunctors by the congruence generated by the following system of identities
\begin{align}\label{CP-ident} 
\sigma_{N}\circ Q_{*}(\alpha_{N})\circ \eta_{N}   =  {\rm Id}_{N}, \hspace{3.15em}\alpha_{N}\circ Q^{*}(\sigma_{N})  =\varepsilon_{N}
\end{align}
satisfied by pairs of morphisms $\alpha_{N}: Q^{*}N\rightarrow N$, $\sigma_{N}: Q_{*}N\xdashrightarrow N$.
 \end{definition}

\begin{example}\label{CPring} For a  unital ring $B$ let us consider a $B$-bimodule $P$ which is finitely generated projective as a right $B$-module. Then we have an  adjunction $Q^{*}\dashv  Q_{*}$ between two endofunctors 
$$Q^{*}:=(-)\otimes_{B}P\ \ \ {\rm and}\ \ \  Q_{*}:=\modB (P, -) \cong (-)\otimes_{B}P^{*}$$
 on the category of right $B$-modules. The unit of that adjunction expresses in terms of the dual basis $\eta: B\rightarrow \End_{B}(P)\cong P\otimes_{B}P^{*}$, $ b\mapsto  bp^{i}\otimes_{B} p^{*}_{i}=p^{i}\otimes_{B} p^{*}_{i}b$ satisfying the property $p=p^{i}\cdot p^{*}_{i}(p)$ (summation over $i$ suppresed). The counit of that adjunction expresses in terms of the evaluation $\varepsilon: P^{*}\otimes_{B} P\rightarrow B$, $p^{*}\otimes_{B}p\mapsto p^{*}(p)$. 
\begin{definition}
With a bimodule $P$ as above, we introduce the following  quotient algebra of the tensor algebra $T_{B}\big( P\oplus P^{*}\big)$ of a $B$-bimodule $P\oplus P^{*}$
\begin{align}\label{Cu-Pims}
\mathbcal{O}_{P}:=T_{B}\big( P\oplus P^{*}\big)/\big( b -\eta(b), p^{*}\otimes_{B}p -\varepsilon(p^{*}\otimes_{B}p)\big).
\end{align}
\end{definition}

Let us note that  \eqref{Cu-Pims} is an algebraic Cuntz-Pimsner ring in the sense of \cite{C-O-11} with respect to the canonical evaluation pairing $\varepsilon: P^{*}\otimes_{B} P\rightarrow B$. 
The next proposition relates this Cuntz-Pimsner ring to our Cuntz-Pimsber monad, justifying our terminology, as follows.
\begin{theorem}\label{C-P}
The Cuntz-Pimsner monad of the adjunction defined by a finitely generated projective from the right $B$-bimodule $P$ as above is isomorphic to the Cuntz-Pimsner $B$-ring $\mathbcal{O}_{P}$.
\end{theorem}
\begin{proof}

After denoting the maps $\alpha_{N}: Q^{*}N\rightarrow N$ and  $\sigma_{N}: Q_{*}N \xdashrightarrow N$ as the right actions $n\otimes_{B}p\mapsto n\triangleleft p$ and $n\otimes_{B}p^{*}\mapsto n\triangleleft p^{*}$, respectively, and next after expressing the unit of the adjunction by the dual basis,  the counit of the adjunction by the evaluation and by applying the dual basis property,  identities \eqref{CP-ident} read  as 
\begin{align}
(n\triangleleft p^{i})\triangleleft p_{i}^{*} =n,\ \ \ \  (n\triangleleft p^{*})\triangleleft p = n\cdot p^{*}(p)
\end{align}
which means that the structure of an object of stable category on $N$ is tantamount to being a right module over $\mathbcal{O}_{P}$. The verification that morphisms in the stable category are equivalent to morphisms of right $\mathbcal{O}_{P}$-modules is routine.
\end{proof}

Finally, given an adjunction between two endofunctors, the following easy monadicity result  relates our Cuntz-Pimsner monad to the aforementioned stable category. This theorem can be regarded as a conceptual justification of somewhat ad hoc definition of the Cuntz-Pimsner monad.
\begin{theorem}\label{C-P}
The Eilenberg-Moore category of the Cuntz-Pimsner monad  is canonically equivalent to the  stable category of  the adjunction. Conversely, forgetting the retraction is monadic and the resulting monad is the Cuntz-Pimsner monad.
\end{theorem}
\begin{proof} 
The presentation of our Cuntz-Pimsner monad in terms of generators and relations \eqref{CP-ident} implies immediately that its Eilenberg-Moore category is isomorphic to the category of objects $N$ equipped with the morphisms $\alpha_{N}: Q^{*}N\rightarrow N$, $\sigma_{N}: Q_{*}N\xdashrightarrow N$ satisfying \eqref{CP-ident} and morphisms $\varphi: N\rightarrow N'$ satisfying  
\begin{align}\label{EM-mor-eqns}
\varphi \circ \alpha_{N} =  \alpha_{N'}\circ Q^{*}(\varphi) , \hspace{1em}\varphi\circ \sigma_{N}  =\sigma_{N'}\circ Q_{*}(\varphi).
\end{align}

Now the argument is that if $\alpha_{N}: Q^{*}N\rightarrow N$ and $
\omega_{N}: N\rightarrow Q_{*}N$ are mates \cite{S-72} under the adjunction $Q^{*}\dashv Q_{*}$, they determine each other by the formulas
 \begin{align}\label{mates}
 \omega_{N} =  Q_{*}(\alpha_{N})\circ \eta_{N} , \hspace{1em}\alpha_{N} =\varepsilon_{N}\circ Q^{*}(\omega_{N}).
\end{align}
Therefore
 \begin{align}\label{st-equiv-CP}
\sigma_{N}\circ  \omega_{N} = \sigma_{N}\circ   Q_{*}(\alpha_{N})\circ \eta_{N} , \qquad\varepsilon_{N}\circ Q^{*}(\omega_{N}\circ \sigma_{N})=\alpha_{N}\circ Q^{*}(\sigma_{N})
\end{align}
what makes the conditions \eqref{stab-eqns} and \eqref{CP-ident} equivalent. and  hence that objects in the stable category and the Eilenberg-Moore category are equivalent. Moreover, by the first equation of \eqref{mates} and naturality of the unit $\eta$ we get
\begin{align}\label{CP-equiv-st}
Q_{*}(\varphi)\circ \omega_{N} 
& =  Q_{*}(\varphi)\circ  Q_{*}(\alpha_{N})\circ \eta_{N}\\
& =  Q_{*}(\varphi \circ  \alpha_{N})\circ \eta_{N}\\
& =  Q_{*}(\alpha_{N'}\circ Q^{*}(\varphi))\circ \eta_{N}\\
& =  Q_{*}(\alpha_{N'})\circ Q_{*}Q^{*}(\varphi)\circ \eta_{N}\\
& =  Q_{*}(\alpha_{N'})\circ \eta_{N'}\circ\varphi \\
&=\omega_{N'}\circ  \varphi, 
\end{align}
while by the second equation of \eqref{mates} and naturality of the counit $\varepsilon$ we get
 \begin{align}\label{st-equiv-CP}
\varphi \circ \alpha_{N} &= \varphi \circ  \varepsilon_{N}\circ Q^{*}(\omega_{N})\\
&=\varepsilon_{N'}\circ Q^{*}Q_{*}(\varphi)\circ Q^{*}(\omega_{N})\\
&=\varepsilon_{N'}\circ Q^{*}(Q_{*}(\varphi)\circ \omega_{N})\\
&=\varepsilon_{N'}\circ Q^{*}(\omega_{N}\circ \varphi )\\
&=\varepsilon_{N'}\circ Q^{*}(\omega_{N})\circ Q^{*}(\varphi )\\
&=\alpha_{N'}\circ Q^{*}(\varphi), 
\end{align}
what proves equivalence of conditions \eqref{stab-mor-eqns} and \eqref{EM-mor-eqns} and  hence that morphisms in the stable category and the Eilenberg-Moore category coincide.

The converse part of the theorem follows from the first part.
\end{proof}

By virtue of this theorem, the notion of the Cuntz-Pimsner monad for an adjunction is therefore a {\bf categorification} of the Cuntz-Pimsner construction in terms of the stable category. 
\end{example}
\subsection{A categorification of the quantum Leavitt path algebra}

\begin{definition}\label{dual-Leavitt} 
For a given representable quantum  quiver $(E, V, s, t)$ with  dicrete (e.g, quasi-finite) the source  morphism $s$,  we have a self-correspondence on $V$ defined by  
two adjunctions $s_{!}\dashv s^{*}$ and $t^{*}\dashv t_{*}$, where  $s_{!}: \cat_{E}\rightarrow \cat_{V}$ and $t_{*}: \cat_{E}\rightarrow \cat_{V}$, as follows. It is easy to see that then we have an adjunction  $Q:=(Q^{*}\dashv Q_{*})$ where $Q_{*}:= t_{*}s^{*}$ and $Q^{*}:=s_{!}t^{*}$. We call the Cuntz-Pimsner monad of that adjunction  the {\bf combinatorial Leavitt path monad} of $(E, V, s, t)$.
\end{definition}

The following example justifies this terminology, when one replaces the category of modules over the algebra with local units, which is generated by orthogonal idempotents  corresponding to vertices of a quiver,  by the opposite category of comodules over the linearized set of vertices.
\begin{example} Let $B$ be a $\Bbbk$-algebra generated by orthogonal idempotents $1_{v}$ corresponding to vertices $v\in V$ of a quiver $(V, E, s. t)$ with the set of edges $E$. Let us define a $B$-bimodule $P$ generated by elements $p^{e}$ corresponding to edges $e\in E$ and the bimodule structure given by
\begin{align}\label{quiv-bimod}
p^e \cdot 1_v =\delta^{s(e)}_{v}p^e,\qquad 1_v\cdot p^e   =\delta^{t(e)}_{v}p^e,
\end{align}
and call it the {\bf algebraic auto-correspondence defined by the quiver}.

\begin{theorem}
Whenever the source map of a classical quiver is quasi-finite, the  Cuntz-Pimsner monad of the algebraic auto-correspondence defined by the quiver is isomorphic to the  the Leavitt path algebra of that quiver.
\end{theorem}

\begin{proof} Let us introduce elements $p^{*}_{e}$ of the right dual $P^{*}$ corresponding to edges of the quiver  by declaring that for any two edges $e, f\in E$
\begin{align}\label{quiv-counit}
p^{*}_{e}(p^f) =\delta^{f}_{e} 1_{s(e)}.
\end{align}
Since elements $p^{e}$ generate $P$ as a right $B$-module,  the following calculation 
\begin{align}\label{dual-basis}
\sum_{e}p^e\cdot p^{*}_{e}(p^f) =\sum_{e}p^e\cdot \delta^{f}_{e}1_{s(e)} = p^f \cdot 1_{s(f)}=p^f.
\end{align}
shows that the system $(p^e, p^{*}_{e})_{e\in E}$ is a dual basis for $P$ regarded as a right $B$-module, proving its projectivity. 

The contragredient $B$-bimodule structure on $P^{*}$ reads as 
\begin{align}\label{contr-bim}
1_v\cdot p^{*}_{e}   =\delta^{s(e)}_{v}p^{*}_{e}, \qquad  p^{*}_{e} \cdot 1_v =\delta^{t(e)}_{v}p^{*}_{e}.
\end{align}
Since for $s$ quasi-finite $P$ is finitely generated projective as a right $B$-module over a ring $B$ with local units, 
\begin{align}\label{P-adj}
Q^{*}:=(-)\otimes_{B}P\quad {\rm and}\quad  Q_{*}:=\modB (P, -) \cong (-)\otimes_{B}P^{*}
\end{align}
is a pair of adjoint functors $Q^{*}\dashv  Q_{*}$. By Theorem \ref{C-P} the corresponding Cuntz-Pimsner monad is isomorphic to $B$-balanced tensoring from the right by the $B$-ring $\mathbcal{O}_{P}$. The verification, based on \eqref{quiv-bimod}, \eqref{quiv-counit} and \eqref{contr-bim}, that $\mathbcal{O}_{P}$ is then the Leavitt path algebra \cite{A-A-SM-17} is left to the reader.

\end{proof}
\end{example}
\begin{remark}
By the definition of $Q_{*}:= t_{*}s^{*}$ and adjunction $t^{*}\dashv t_{*}$, the morphism $\omega_{N}: N\rightarrow Q_{*}N$ is equivalent to a morphism $t^{*}N\rightarrow s^{*}N$ and hence can be regarded as a representation of the quiver \cite{G-72} which is completed to a retraction $(\omega_{N}, \sigma_{N})$ leaving the counit of the adjunction $Q^{*}\dashv Q_{*}$ stable, and hence can be called {\bf stable representation} of the quiver. Moreover, in a given stable representation of a quiver {\bf forgetting about arrows} and restricting the representation to vertices only forgets also about stability.  In this terminology, Theorem \ref{C-P} specializes to the following result.
\begin{theorem}\label{LP}
The Eilenberg-Moore category of the Leavitt path monad of a (quantum) quiver  is canonically equivalent to the category of stable representations of that quiver. Conversely, forgetting about arrows in a stable representation of the quiver is monadic and the resulting monad is the Leavitt path monad.
\end{theorem} 
\end{remark}

\begin{remark}
 Let us stress that when defining our $Q$-stable category of a quantum quiver, in opposite to the construction of a graph C*-algebra of a quiver,  the Leavitt path algebra of a quiver as in \cite{C-O-11} or of a quantum quiver in the sense of \cite{G-G-P-25},  instead of the data like an algebra-valued scalar product of a Hilbert-module, pairing of bimodules or adjacency matrices or their quantum generalizations, we use the machinery of adjunctions which is mathematically more fundamental.  Note that in  the coalgebraic version  we can weaken  the row-finite condition (of finite emission) appearing in the algebraic Cuntz-Krieger family in the construction of the absolute Cuntz-Pimsner ring. 
\end{remark}

\section*{Acknowledgement}
This research is part of the EU Staff Exchange project 101086394 Operator Algebras That One
Can See. The project is co-financed by the Polish Ministry of Education and Science under
the program PMW (grant agreement 5448/HE/2023/2).

\noindent The author sincerely acknowledges Jakub Zarzycki’s contribution to the study of quantum Boolean algebras.

\end{document}